\documentclass[11pt,a4paper,onesided]{article}

\usepackage{ amssymb, amsmath, amsthm, graphics, xspace, enumerate}
\usepackage{graphicx}
\usepackage[a4paper,colorlinks=true,citecolor=blue,urlcolor=blue,linkcolor=blue,bookmarksopen=true,unicode=true,pdffitwindow=true]{hyperref}
 \hypersetup{pdfauthor={Nikola Sandri\'{c}}}
\hypersetup{pdftitle={Ergodic property  of stable-like Markov
chains}}

\newtheorem{tm}{tm}[section]
\newtheorem{theorem}[tm]{Theorem}

\newtheorem{corollary}[tm]{Corollary}

\newtheorem{definition}[tm]{Definition}

\newcommand {\R} {\ensuremath{\mathbb{R}}}
\newcommand {\ZZ} {\ensuremath{\mathbb{Z}}}
\newcommand {\N} {\ensuremath{\mathbb{N}}}
\newcommand {\CC} {\ensuremath{\mathbb{C}}}

\newcommand{\chain}[1]{\{#1_n\}_{n\geq0}}
\numberwithin{equation}{section}
\def\be{\begin{equation}}
\def\ee{\end{equation}}
\begin{document}

 \title{Ergodic property of stable-like Markov chains}
 \author{Nikola Sandri\'{c}\\ Department of Mathematics\\
         Faculty of Civil Engineering, University of Zagreb\\
         Fra Andrije Ka\v{c}i\'{c}a-Mio\v{s}i\'{c}a 26, 10000  Zagreb,
         Croatia \\
        Email: nsandric@grad.hr }

 \maketitle
\begin{center}
{
\medskip

} \end{center}

\begin{abstract}
A stable-like Markov chain is a time-homogeneous  Markov chain on
the real line with the transition kernel $p(x,dy)=f_x(y-x)dy$, where
the density functions $f_x(y)$, for large $|y|$, have a power-law
decay with exponent $\alpha(x)+1$, where $\alpha(x)\in(0,2)$. In
this paper, under a certain uniformity condition on the density
functions $f_x(y)$ and additional mild drift conditions, we give
 sufficient conditions for recurrence
 in the case when $0<\liminf_{|x|\longrightarrow\infty}\alpha(x)$,
 sufficient conditions for transience  in the case when
 $\limsup_{|x|\longrightarrow\infty}\alpha(x)<2$ and sufficient
 conditions for ergodicity in the case when
 $0<\inf\{\alpha(x):x\in\R\}$.
 As a special case of these results, we  give a new
proof for the recurrence and transience property of a symmetric
$\alpha$-stable random walk on $\R$ with the index of stability
$\alpha\neq1.$

\end{abstract}
{\small \textbf{AMS 2010 Mathematics Subject Classification:} 60J05,
60G52} \smallskip

\noindent {\small \textbf{Keywords and phrases:} ergodicity,
Foster-Lyapunov drift criteria,    recurrence, stable distribution,
stable-like Markov chain, transience}

%
%
%
%


\section{Introduction}

\quad \ \ Let $(\Omega,\mathcal{F},\mathbb{P})$ be a probability
space and let $\{J_n\}_{n\geq1}$ be a sequence of i.i.d.  random
variables on $(\Omega,\mathcal{F},\mathbb{P})$ taking values in
$\R^{d}$, $d\geq1$. Let us define $X_n:=\sum_{i=1}^{n}J_i$ and
$X_0:=0$. The sequence $\{X_n\}_{n\geq0}$ is called a \emph{random
walk} with jumps $\{J_n\}_{n\geq1}$. The random walk
$\{X_n\}_{n\geq0}$ is said to be \emph{recurrent} if
$$\sum_{n=1}^{\infty}\mathbb{P}\left(|X_n|<a\right)=\infty\quad\textrm{for every}\ a>0 ,$$ and  \emph{transient} if $$\sum_{n=1}^{\infty}\mathbb{P}\left(|X_n|<a\right)<\infty\quad\textrm{for every}\ a>0.$$
It is well known that every random walk is either recurrent or
transient (see   \cite[Theorem 4.2.6]{durrett}). Further, recall
that an $\R^{d}$-valued random variable $S$ is said to have
\emph{stable distribution} if, for any $n\in\N$, there are $a_n
> 0$ and $b_n\in\R^{d}$, such that
$$S_1+\ldots+S_n\stackrel{\hbox{\scriptsize d}}{=}a_nS+b_n,$$ where
$S_1,\ldots, S_n$ are independent copies of $S$ and
$\stackrel{\hbox{\scriptsize d}}{=}$ denotes equality in
distribution.  It turns out that $a_n=n^{\frac{1}{\alpha}}$ for some
$\alpha\in(0,2]$ which is called the index of stability (see
\cite[Definition 1.1.4 and Corollary 2.1.3]{taqqu}).  The case when
$\alpha=2$ corresponds to the Gaussian random variable. A random
walk $\{X_n\}_{n\geq0}$ is said to be \emph{stable} if the random
variable $J_1$ has stable distribution.  In the case $d=1$, every
stable distribution is characterized by four parameters: the
stability parameter $\alpha\in(0,2],$  the skewness parameter
$\beta\in[-1,1]$, the scale parameter $\gamma\in(0,\infty)$ and the
shift parameter $\delta\in\R$ (see \cite[Definition 1.1.6]{taqqu}).
Using the notation from \cite{taqqu}, we denote  one-dimensional
stable distributions by $S_{\alpha}(\beta,\gamma,\delta).$ For
symmetric one-dimensional stable distributions, that is, for
$S_{\alpha}(0,\gamma,0)$ (see \cite[Property 1.2.5]{taqqu}), we
write $S\alpha S$. It is well known that a $S\alpha S$ random walk
is recurrent if and only if $\alpha\geq1$ (see the discussion after
\cite[Lemma 4.2.12]{durrett}). For the occurrence of stable
distributions in applications (for example, biology, economy,
engineering, physics, etc.) we refer the reader to
\cite{feller-book} and \cite{zolotarev-book}. The concept of
stable-like Markov chains, as a natural generalization of
$S_{\alpha}(0,\gamma,\delta)$ random walks,
 has been introduced in \cite{sandric-rectrans} in the way
  that the  stability,  scale and shift parameters  of the jump distribution depend
on the current position of the process. In the same paper, the
author also analyzes structural properties and provides sufficient
conditions for  recurrence and transience of such processes. In this
paper, we generalize the results presented in
\cite{sandric-rectrans} and we also discuss  ergodic property of
such processes.

From now on, let us denote by $\mathcal{B}(\R^{d})$ the Borel
$\sigma$-algebra on $\R^{d}$, $d\geq1$, by $\lambda(\cdot)$ the
Lebesgue measure on $\mathcal{B}(\R^{d})$ and for arbitrary
$B\in\mathcal{B}(\R^{d})$ and $x\in\R^{d}$ we define
$B-x:=\{y-x:y\in B\}$. Furthermore, for $f,g:\R\longrightarrow\R$,
let us introduce the notation $f(y)\sim g(y),$ as $y\longrightarrow
y_0$, for $\lim_{y\longrightarrow y_0}f(y)/g(y)=1,$ where
$y_0\in[-\infty,\infty]$.  Recall that if $f(y)$ is the density
function of  a $S_{\alpha}(0,\gamma,\delta)$ distribution, where
$\alpha\in(0,2)$, $\gamma\in(0,\infty)$ and $\delta\in\R$, then
$$f(y)\sim c_{\alpha}|y|^{-\alpha-1},$$ for
$|y|\longrightarrow\infty,$ where $c_1=\gamma/2$ and
$c_{\alpha}=\gamma\Gamma(\alpha+1)\sin\left(\pi\alpha/2\right)/\pi,$
for $\alpha\neq1$ (see    \cite[Property 1.2.15]{taqqu}).

 Now, we recall the definition of stable-like Markov chains introduced
 in \cite{sandric-rectrans}.
Let $\alpha:\R\longrightarrow(0,2)$  be an arbitrary function and
let   $\{f_x:x\in\R\}$ be a family of density functions on $\R$ and
$c:\R\longrightarrow(0,\infty)$ such that
\begin{description}
  \item [\textbf{(C1)}] $x\longmapsto f_x(y)$ is a Borel measurable function for
all $y\in\R$
  \item [\textbf{(C2)}] $f_x(y)\sim
c(x)|y|^{-\alpha(x)-1},$ for $|y|\longrightarrow\infty,$ for all
$x\in\R$
 \item [\textbf{(C3)}] there exists $k_0>0$ such that
 $$\lim_{|y|\longrightarrow\infty}\sup_{x\in[-k_0,k_0]^{c}}\left|f_x(y)\frac{|y|^{\alpha(x)+1}}{c(x)}-1\right|=0$$
                          \item [\textbf{(C4)}] $\displaystyle\inf_{x\in C}c(x)>0$
                          for every compact set
                          $C\subseteq[-k_0,k_0]^{c}$
                          \item [\textbf{(C5)}]  there exists $l_0>0$ such  that for every compact set $C\subseteq[-l_0,l_0]^{c}$ with $\lambda(C)>0$,
                          we have
$$\inf_{x\in[-k_0,k_0]}\int_{C-x}f_x(y)dy>0.$$
\end{description}
Let us define a Markov chain  $\{X_n\}_{n\geq0}$ on $\R$ by the
following  transition kernel \be \label{eq:1.1}
p(x,dy):=f_x(y-x)dy.\ee The chain jumps from the state $x$ with the
transition density $f_x(y-x)$ with  the power-law decay with
exponent $\alpha(x)+1$, and this jump distribution  depends only on
the current state $x$. Transition densities $\{f_x:x\in\R\}$ are
asymptotically equivalent to the densities of
$S_{\alpha}(0,\gamma,\delta)$ distributions, and such Markov chain
 is called a \emph{stable-like Markov chain}.

Conditions (C1)-(C5) are needed to control  the jumps of
$\{X_n\}_{n\geq0}$. They are crucial in deriving certain structural
properties of stable-like chains such as irreducibility and
aperiodicity and  in identifying the class of ``singletons" for such
chains. These properties are essential in finding sufficient
conditions for recurrence, transience and ergodicity. We refer the
reader to \cite{sandric-rectrans} for more details about conditions
(C1)-(C5).

A concrete application of stable-like processes in geophysics is
given in \cite{geo}. Shortly, paleoclimatic records from ice core
show that the climate of the last glacial period experienced rapid
transitions between two climatic states and the triggering mechanism
for climate changes is random fluctuations of the atmospheric
forcing on the ocean circulation. To describe this stochastic
climate dynamics the Langevin equation $dy=-(dU/dy)dt+dN$ is used.
The variable $y$ represents the climate state associated with the
pole ward heat transport. The  first term on the right hand side
represents the dynamics of the ocean circulation, where $U$ is the
climate potential which describes  the multi state character of the
climate system, and the second term on the right hand side is a
noise term which represents the atmospheric forcing on the climate
state (wind stress, heating and water transport). From the data from
the Greenland Ice Core Project, in \cite{geo} has been deduced that
this noise has an $\alpha$-stable component which
 depends  on the position (climate state).
This suggests that the  noise should be modeled by a stable-like
process.

The aim of this paper is to investigate a long-time behavior of
stable-like chains, that is, as in the random walk case, to find
conditions for recurrence, transience and ergodicity  of stable-like
chains in terms of the functions $\alpha(x)$ and $c(x)$. To the best
of our knowledge, all methods used in establishing conditions for
recurrence and transience in the random walk case are based on the
i.i.d. property of random walk jumps, that is, laws of large numbers
(Chung-Fuchs theorem), central limit theorems, characteristic
functions approach (Stone-Ornstein formula), etc. (see
\cite[Theorems 4.2.7, 4.2.8 and 4.2.9]{durrett}). Although we deal
with distributions similar to $S_{\alpha}(\beta,\gamma,\delta)$
distributions, it is not clear if these methods can be used in the
case of the non-constant functions $\alpha(x)$ and $c(x)$. Special
cases of this problem have been considered in
\cite{rogozin-foss-oscrw}, \cite{sandric-periodic} and
  \cite{sandric-rectrans}. In \cite{rogozin-foss-oscrw}  the authors consider the countable
state space $\ZZ$ and the function $\alpha(x)$ is a two-valued step
function which takes one value on negative integers and the other
one on nonnegative integers. In \cite{sandric-periodic}, the author
considers two
 special cases of stable-like chains, as in \cite{rogozin-foss-oscrw}, the case when the
 function $\alpha(x)$ is of the form $$\alpha(x)=\left\{\begin{array}{cc}
                                                      \alpha, & x<0 \\
                                                      \beta,&
                                                      x\geq0,
                                                    \end{array}\right.$$
                                                    and the case
                                                    when $\alpha(x)$
                                                    is periodic
                                                    function. In the
                                                    first case, it
                                                    has been proved
                                                    that the
                                                    corresponding
                                                    stable-like
                                                    chain is
                                                    recurrent if and
                                                    only if
                                                    $\alpha+\beta\geq2$,
                                                    while in the
                                                    second case, under the assumption $\lambda(\{x:\alpha(x)=\alpha_0:=\inf \{\alpha(y):y\in\R\}\})>0$, it has been proved that the
                                                    corresponding
                                                    stable-like
                                                    chain is
                                                    recurrent if and
                                                    only if
                                                    $\alpha_0\geq1.$
In \cite{sandric-rectrans}, it has been proved that
$\liminf_{|x|\longrightarrow\infty}\alpha(x)>1$, under an additional
mild drift condition,  implies recurrence of the corresponding
stable-like chain, and
 $\limsup_{|x|\longrightarrow\infty}\alpha(x)<1$, under an additional mild drift
 condition,
 implies the transience of the corresponding stable-like chain. In this
 paper, we generalize these results. More precisely,   we give
 sufficient conditions for  recurrence
 in the case when $0<\liminf_{|x|\longrightarrow\infty}\alpha(x)$
 and
 sufficient conditions for transience  in the case when
 $\limsup_{|x|\longrightarrow\infty}\alpha(x)<2$. Further, we also discuss the ergodic property of stable-like chains and give   sufficient
 conditions for ergodicity in the case when
 $0<\inf\{\alpha(x):x\in\R\}$.
Let us also mention that if we allow $\alpha(x)\in(0,\infty)$, the
recurrence property of the corresponding stable-like chain in the
case when $\liminf_{|x|\longrightarrow\infty}\alpha(x)>2$ has been
covered in \cite{rus}. For the continuous-time version of
stable-like chains and their recurrence and transience property we
refer the reader to \cite{bjoern-overshoot}, \cite{franke-periodic,
franke-periodicerata}, \cite{rene-wang-feller}, \cite{sandric-spa}
and \cite{sandric-tams}.

Before stating the main results of this paper, we recall relevant
definitions of recurrence, transience and ergodicity.
\begin{definition} Let $\chain{Y}$ be a Markov
chain on $(\R^{d},\mathcal{B}(\R^{d}))$. The chain $\chain{Y}$ is
called
\begin{enumerate}
                      \item [(i)] \emph{Lebesgue irreducible} if
                      $\lambda(B)>0$ implies
                      $\sum_{n=1}^{\infty}\mathbb{P}(Y_n\in B|Y_0=x)>0$ for all
                      $x\in\R^{d}$.
                      \item [(ii)] \emph{Recurrent} if it is Lebesgue
                      irreducible and if $\lambda(B)>0$ implies $\sum_{n=1}^{\infty}\mathbb{P}(Y_n\in B|Y_0=x)=\infty$ for all
                      $x\in\R^{d}$.
                      \item[(iii)] \emph{Harris recurrent} if it is Lebesgue
                      irreducible and if $\lambda(B)>0$ implies $\mathbb{P}^{x}(\tau_B<\infty)=1$ for all
                      $x\in\R^{d}$, where $\tau_B:=\min\{n\in\N:Y_n\in
                      B\}.$
\item [(iv)] \emph{Transient} if it is Lebesgue
                      irreducible and if there exists a countable
                      covering of $\R^{d}$ with  sets
$\{B_j\}_{j\in\N}\subseteq\mathcal{B}(\R^{d})$, such that for each
$j\in\N$ there is a finite constant $M_j\geq0$ such that
$\sum_{n=1}^{\infty}\mathbb{P}(Y_n\in B_j|Y_0=x)\leq M_j$ holds for
all $x\in\R^{d}$.
 \end{enumerate}
 \end{definition}
Note that the Lebesgue irreducibility of
$S_{\alpha}(0,\gamma,\delta)$ random walks  is trivially satisfied,
and  the Lebesgue irreducibility of
 general stable-like chains has been shown in \cite[Proposition 2.1]{sandric-rectrans}. Hence, according to \cite[Theorem
8.3.4]{meyn-tweedie-book}, every stable-like chain is either
recurrent or transient. Further, clearly, every Harris recurrent
chain is recurrent but in general, these two properties are not
equivalent. They differ on a set of Lebesgue measure zero (see
\cite[Theorem 9.1.5]{meyn-tweedie-book}). In the case of stable-like
chains, these two properties are equivalent (see \cite[Proposition
5.3]{sandric-rectrans}). For further structural properties of
stable-like chains we refer the reader to \cite{sandric-rectrans}. A
$\sigma$-finite measure $\pi(\cdot)$ on $\mathcal{B}(\R^{d})$ is
called an \emph{invariant measure} for a Markov chain $\chain{Y}$ on
$(\R^{d},\mathcal{B}(\R^{d}))$ if
$$\pi(B)=\int_{\R}\mathbb{P}(Y_{1}\in B|Y_0=x)\pi(dx)$$ holds for all  $B\in\mathcal{B}(\R^{d}).$
\begin{definition}
 A Markov chain $\chain{Y}$ on $(\R^{d},\mathcal{B}(\R^{d}))$ is called \emph{ergodic} if an
invariant probability measure $\pi(\cdot)$ exists and if
$$\lim_{n\longrightarrow\infty}||\mathbb{P}(Y_n\in\cdot|Y_0=x)-\pi(\cdot)||=0$$ holds for
all $x\in\R^{d},$ where $||\cdot||$ denotes the total variation norm
on the space of signed measures.
\end{definition}
Now, let us state the main results of this paper. The following
three constants will appear in the statements of the main results.
For $\alpha\in(0,2)$ let
$$R_1(\alpha):=\frac{-\pi\mathrm{ctg}\left(\frac{\pi
\alpha}{\rm{2}}\right)}{\alpha},$$ for $\alpha\in(0,2)$ and
$\beta\in(0,1]\cap(0,\alpha)$  let
$$R_2(\alpha,\beta):=-\sum_{n=1}^{\infty}{\beta \choose
2n}\frac{2}{2n-\alpha}+\frac{2}{\alpha}-\frac{_2F_1(-\beta,\alpha-\beta,1+\alpha-\beta;-1)+\
_2F_1(-\beta,\alpha-\beta,1+\alpha-\beta;1)}{\alpha-\beta}$$ and for
$\alpha\in(0,2)$ and $\beta\in(0,1)$ let
$$T(\alpha,\beta):=\sum_{n=1}^{\infty}{-\beta \choose
2n}\frac{2}{2n-\alpha}-\frac{2}{\alpha}+\frac{_2F_1\left(\beta,\alpha+\beta,1+\alpha+\beta;1\right)+\,
_2F_1\left(\beta,\alpha+\beta,1+\alpha+\beta;-1\right)}{\alpha+\beta},$$
where  ${z \choose n}$ is the binomial coefficient and
$_2F_1(a,b,c;z)$ is the Gauss hypergeometric function  (see Section
3 for the definition of this function). Clearly, the constant
$R_1(\alpha)$, as a function of $\alpha\in(0,2)$, is strictly
increasing and it satisfies
$\lim_{\alpha\longrightarrow0}R_1(\alpha)=-\infty$, $R_1(1)=0$ and
$\lim_{\alpha\longrightarrow2}R_1(\alpha)=\infty$. The constant
$R_2(\alpha,\beta)$, as a function of $\alpha\in(\beta,2)$ for fixed
$\beta\in(0,1)$, is strictly increasing and it satisfies
$\lim_{\alpha\longrightarrow\beta}R_2(\alpha,\beta)=-\infty$,
$R_2(1+\beta,\beta)=0$ and
$\lim_{\alpha\longrightarrow2}R_2(\alpha,\beta)=\infty$ (see the
proof of Theorem \ref{tm1.1}). Finally, the constant
$T(\alpha,\beta)$, as a function of $\alpha\in(0,2)$ for fixed
$\beta\in(0,1)$, is strictly increasing and it satisfies
$\lim_{\alpha\longrightarrow0}T(\alpha,\beta)=-\infty$,
$T(1-\beta,\beta)=0$ and
$\lim_{\alpha\longrightarrow2}T(\alpha,\beta)=\infty$ (see the proof
of Theorem \ref{tm1.2}).
\begin{theorem}\label{tm1.1}
Let $\alpha:\R\longrightarrow(0,2)$ be an arbitrary function such
that
$$0<\alpha:=\displaystyle\liminf_{|x|\longrightarrow\infty}\alpha(x)$$
and let $\beta\in(0,1]\cap(0,\alpha)$ be arbitrary. Furthermore, let
$\{f_x:x\in\R\}$ be a family of density functions on $\R$ and
$c:\R\longrightarrow(0,\infty)$  which satisfy conditions (C1)-(C5)
and
 such that
\begin{align}\label{eq:1.2}\limsup_{\delta\longrightarrow0}\limsup_{|x|\longrightarrow\infty}\frac{|x|^{\alpha(x)}}{c(x)}\int_{-\delta|x|}^{\delta|x|}\log\left(1+\mathrm
{sgn}\it\, (x)\frac{y}{\rm{1}+\it{|x|}}\right)f_{x}(y)dy<R_1(\alpha)
\end{align}
or
\begin{align}\label{eq:1.3}\limsup_{\delta\longrightarrow0}\limsup_{|x|\longrightarrow\infty}\frac{|x|^{\alpha(x)}}{c(x)}\int_{-\delta|x|}^{\delta|x|}\left(\left(1+\mathrm
{sgn}\it\,
(x)\frac{y}{\it{|x|}}\right)^{\beta}-1\right)f_{x}(y)dy<R_2(\alpha,\beta).
\end{align}
 Then the stable-like Markov chain $\{X_n\}_{n\geq0}$     is
recurrent. \end{theorem} \begin{theorem}\label{tm1.2} Let
$\alpha:\R\longrightarrow(0,2)$ be an arbitrary function such that
$$\displaystyle\limsup_{|x|\longrightarrow\infty}\alpha(x)=:\alpha<2$$
and let $\beta\in(0,1)$ be arbitrary.
 Furthermore,  let
$\{f_x:x\in\R\}$ be a family of density functions on $\R$ and
$c:\R\longrightarrow(0,\infty)$  which satisfy conditions (C1)-(C5)
and
 such that
\be\label{eq:1.4}\liminf_{\delta\longrightarrow0}\liminf_{|x|\longrightarrow\infty}\frac{|x|^{\alpha(x)}}{c(x)}\int_{-\delta|x|}^{\delta|x|}\left(1-\left(1+\mathrm
{sgn}\it\,
(x)\frac{y}{\rm{1}+|\it{x}|}\right)^{-\beta}\right)f_{x}(y)dy>T(\alpha,\beta).\ee
 Then the stable-like Markov chain
$\{X_n\}_{n\geq0}$    is transient.
\end{theorem}

\begin{theorem}\label{tm1.3} Let $\alpha:\R\longrightarrow(0,2)$ be an arbitrary function such
that
$$0<\inf\{\alpha(x):x\in\R\},$$ let
$\alpha:=\liminf_{|x|\longrightarrow\infty}\alpha(x)$ and let
$\beta\in(0,1]\cap(0,\inf\{\alpha(x):x\in\R\})$ be arbitrary.
Furthermore, let $\{f_x:x\in\R\}$ be a family of density functions
on $\R$ and $c:\R\longrightarrow(0,\infty)$  which satisfy
conditions (C1)-(C5) and
 such that
\begin{align}\label{eq:1.5}\limsup_{d\longrightarrow0}\limsup_{\delta\longrightarrow0}\limsup_{|x|\longrightarrow\infty}\frac{|x|^{\alpha(x)}}{c(x)}\left(\int_{-\delta|x|}^{\delta|x|}\log\left(1+\mathrm
{sgn}\it\,
(x)\frac{y}{\rm{1}+\it{|x|}}\right)f_{x}(y)dy+d\right)<R_1(\alpha)
\end{align}
or
\begin{align}\label{eq:1.6}\limsup_{d\longrightarrow0}\limsup_{\delta\longrightarrow0}\limsup_{|x|\longrightarrow\infty}\frac{|x|^{\alpha(x)}}{c(x)}\left(\int_{-\delta|x|}^{\delta|x|}\left(\left(1+\mathrm
{sgn}\it\,
(x)\frac{y}{\it{|x|}}\right)^{\beta}-1\right)f_{x}(y)dy+d|x|^{-\beta}\right)<R_2(\alpha,\beta).
\end{align}
Then the stable-like Markov chain $\{X_n\}_{n\geq0}$     is ergodic.
\end{theorem}
Conditions (\ref{eq:1.2}), (\ref{eq:1.3}),  (\ref{eq:1.4}),
(\ref{eq:1.5}) and (\ref{eq:1.6}) are needed to control small jumps
of  $\{X_n\}_{n\geq0}$. Big jumps are controlled with condition
(C3).  If we assume that
$$\limsup_{|x|\longrightarrow\infty}\alpha(x)<2\quad\textrm{
and}\quad
\lim_{|x|\longrightarrow\infty}c(x)|x|^{2-\alpha(x)}=\infty,$$ then
conditions (\ref{eq:1.2}) and (\ref{eq:1.3}) are equivalent to
\begin{align}\label{eq:1}\limsup_{\delta\longrightarrow0}\limsup_{|x|\longrightarrow\infty}
\mathrm
{sgn}(\it{x})\frac{|\it{x}|^{\alpha(\it{x})-\rm{1}}}{c(\it{x})}\int_{-\delta|\it{x}|}^{\delta|\it{x}|}yf_{x}(y)dy<R_{\rm{1}}(\alpha),\end{align}
condition (\ref{eq:1.4}) is equivalent to
\begin{align}\label{eq:2}\liminf_{\delta\longrightarrow0}\liminf_{|x|\longrightarrow\infty}
\mathrm
{sgn}(\it{x})\frac{|\it{x}|^{\alpha(\it{x})-\rm{1}}}{c(\it{x})}\int_{-\delta|\it{x}|}^{\delta|\it{x}|}yf_{x}(y)dy>R_{\rm{1}}(\alpha)\end{align}
and conditions (\ref{eq:1.5}) and (\ref{eq:1.6}) are equivalent to
\begin{align}\label{eq:3}\limsup_{d\longrightarrow0}\limsup_{\delta\longrightarrow0}\limsup_{|x|\longrightarrow\infty}
\frac{|\it{x}|^{\alpha(\it{x})-\rm{1}}}{c(\it{x})}\left(\mathrm
{sgn}(\it{x})\int_{-\delta|\it{x}|}^{\delta|\it{x}|}yf_{x}(y)dy+d|\it{x}|\right)<R_{\rm{1}}(\alpha)\end{align}
and
\begin{align}\label{eq:4}\limsup_{\beta\longrightarrow0}\limsup_{d\longrightarrow0}\limsup_{\delta\longrightarrow0}\limsup_{|x|\longrightarrow\infty}\frac{|x|^{\alpha(x)-\rm{1}}}{c(x)}\left(\mathrm
{sgn}(\it{x})\int_{-\delta|x|}^{\delta|x|}yf_{x}(y)dy+\frac{d|x|^{-\beta+\rm{1}}}{\beta}\right)<R_{\rm{1}}(\alpha),\end{align}
respectively. In addition, if
$\liminf_{|x|\longrightarrow\infty}\alpha(x)>1$, then the term
$\int_{-\delta|x|}^{\delta|x|}yf_{x}(y)dy$ in (\ref{eq:1}),
(\ref{eq:2}), (\ref{eq:3}) and (\ref{eq:4})   can be replaced by
$\mathbb{E}[X_1-X_0|X_0=x]=\int_{\R}yf_{x}(y)dy$. Hence, conditions
(\ref{eq:1.2}), (\ref{eq:1.3}), (\ref{eq:1.5}) and (\ref{eq:1.6})
actually say that when the stable-like chain $\{X_n\}_{n\geq0}$ has
moved far away from the origin it cannot have strong tendency to
move further from the origin, while condition (\ref{eq:1.4}) says
that the stable-like chain $\{X_n\}_{n\geq0}$ has a permanent
tendency of moving away from the origin.
 See Section 4 for
the proof of the above equivalences and further discussion about
 conditions (\ref{eq:1.2}),
(\ref{eq:1.3}),  (\ref{eq:1.4}), (\ref{eq:1.5}) and (\ref{eq:1.6}).

As a simple consequence of these results we get a new proof for
 the well-known recurrence and transience property of $S\alpha S$ random walks.
\begin{corollary}\label{c1.4}
A $S\alpha S$ random walk with $1<\alpha\leq2$
 is recurrent. A $S_{\alpha}(0,\gamma,\delta)$ random walk with $0 < \alpha < 1$ and arbitrary shift $\delta\in\R$  is transient.
\end{corollary}
Note that Theorem \ref{tm1.1} and condition (\ref{eq:1}), that is,
conditions (\ref{eq:1.2}) and (\ref{eq:1.3}), do not imply the
recurrence property of S$1$S random walk since, in this case, the
left-hand side and the right-hand side in (\ref{eq:1}) are
 equal to zero. A simple example of the application of Theorems \ref{tm1.1},
 \ref{tm1.2} and \ref{tm1.3} in the non-random walk case is the
 following. Let $\alpha:\R\longrightarrow(0,2)$,
 $\gamma:\R\longrightarrow(0,\infty)$ and
 $\delta:\R\longrightarrow\R$ be arbitrary Borel measurable
 functions which take finitely many values. Then, the Markov
 chain $\{X_n\}_{n\geq0}$ with $S_{\alpha(x)}(0,\gamma(x),\delta(x))$ jumps satisfies
 conditions (C1)-(C5), that is, it is a stable-like chain. Now, by Theorems \ref{tm1.1},
 \ref{tm1.2} and \ref{tm1.3},
 \begin{enumerate}
   \item [(i)] if $\inf\{\alpha(x):x\in\R\}>1$,
   $\sup\{\delta(x):x>0\}\leq0$ and $\inf\{\delta(x):x<0\}\geq0$,
   then, from (\ref{eq:1}), $\{X_n\}_{n\geq0}$ is recurrent.
   \item [(ii)] if $\sup\{\alpha(x):x\in\R\}<1$,
   then, from (\ref{eq:2}), $\{X_n\}_{n\geq0}$ is transient.
   \item [(iii)] if $\inf\{\alpha(x):x\in\R\}>1$,
   $\sup\{\delta(x):x>0\}<0$ and $\inf\{\delta(x):x<0\}>0$,
   then, from (\ref{eq:1.6}) where we take $\beta=1$, $\{X_n\}_{n\geq0}$ is ergodic.
 \end{enumerate}

 Now, we explain our strategy of  proving the  main results. The
proofs of Theorems~\ref{tm1.1}, \ref{tm1.2} and~\ref{tm1.3} are
based on the \emph{Foster-Lyapunov drift  criteria}  (see
\cite[Theorems 8.4.2, 8.4.3 and 13.0.1]{meyn-tweedie-book}). These
criteria are based on finding an appropriate ``distance" function
$V(x)$ (positive and unbounded in the recurrence case, positive and
bounded in the transience case
 and positive and finite in the  ergodic case)
and
 a compact set
$C\subseteq\R$, such that $\Delta
V(x):=\mathbb{E}[V(X_1)-V(X_0)|X_0=x]\leq0$, in the recurrent case,
$\Delta V(x)\geq0$, in the transient case,
 and $\Delta
V(x)\leq-d$, for some $d>0$, in the ergodic case,  for every $x\in
C^{c}.$ The idea is to find test functions $V(x)$ such that the
associated level sets $C_V(r):=\{y:V(y)\leq r\}$ are ``singletons"
 and such that $C_V(r)\uparrow\R$, for
$r\longrightarrow\infty$, in the cases of recurrence and ergodicity,
and $C_V(r)\uparrow\R$, for $r\longrightarrow 1$, in the case of
transience. In the recurrent case, for the test function we take
$\log(1+|x|)$ and $|x|^{\beta}$, where
$\beta\in(0,1]\cap(0,\liminf_{|x|\longrightarrow\infty}\alpha(x))$
is arbitrary.
 In the transient case, for the test function we take
$V(x)=1-(1+|x|)^{-\beta}$, where $\beta\in(0,1)$ is arbitrary, and
in the ergodic case we take again $\log(1+|x|)$ and $|x|^{\beta}$,
where $\beta\in(0,1]\cap(0,\inf\{\alpha(x):x\in\R\})$ is arbitrary.
Now, by proving
 that
 $$\limsup_{|x|\longrightarrow\infty}\frac{|x|^{\alpha(x)}}{c(x)}\Delta
V(x)<0\quad\textrm{and}\quad\limsup_{|x|\longrightarrow\infty}\frac{|x|^{\alpha(x)-\beta}}{c(x)}\Delta
V(x)<0$$ in the recurrent case,
$$\liminf_{|x|\longrightarrow\infty}\frac{|x|^{\alpha(x)+\beta}}{c(x)}\Delta
V(x)>0$$ in the transient case and
$$\limsup_{|x|\longrightarrow\infty}\frac{|x|^{\alpha(x)}}{c(x)}\left(\Delta
V(x)+d\right)<0\quad\textrm{and}\quad\limsup_{|x|\longrightarrow\infty}\frac{|x|^{\alpha(x)-\beta}}{c(x)}\left(\Delta
V(x)+d\right)<0,$$ for some $d>0$, in the ergodic case, the proofs
of Theorems~\ref{tm1.1}, \ref{tm1.2} and~\ref{tm1.3} are
accomplished.

Let us remark that a similar approach,  using similar test
functions, can be found in \cite{lamperti}, \cite{rus} and
\cite{sandric-rectrans} in the discrete-time case and in
\cite{stramer}, \cite{wang-ergodic} and \cite{sandric-spa}  in the
continuous-time case.

The paper is organized as follows. In Section 2  we discuss several
structural properties of stable-like chains  which will be crucial
in finding sufficient conditions for   recurrence, transience and
ergodicity. In Section 3, using the Foster-Lyapunov drift criteria,
we give the proofs of Theorems \ref{tm1.1}, \ref{tm1.2} and
\ref{tm1.3}. Finally, in Section 4, we  discuss conditions
(\ref{eq:1.2}), (\ref{eq:1.3}), (\ref{eq:1.4}), (\ref{eq:1.5}) and
(\ref{eq:1.6}) and some consequences of the main results.

Throughout the paper we use the following notation.  We write
$\ZZ_+$ and $\ZZ_-$ for nonnegative  and nonpositive integers,
respectively. For $x,y\in\R$ let $x\wedge y :=\min\{x,y\}$ and
$x\vee y :=\max\{x,y\}$. Furthermore, $\{X_n\}_{n\geq0}$ will denote
a stable-like Markov chain  given by (\ref{eq:1.1}) with transition
densities satisfying conditions (C1)-(C5), while $\{Y_n\}_{n\geq0}$
will denote an arbitrary Markov chain on
$(\R^{d},\mathcal{B}(\R^{d}))$ given by the transition kernel
$p(x,B)$. For $x\in\R^{d}$, $B\in \mathcal{B}(\R^{d})$ and $n\in\N$
let $p^{n}(x,B):=\mathbb{P}^{x}(Y_n\in B):=\mathbb{P}(Y_n\in
B|Y_0=x)$.

\section{Preliminary and auxiliary results}

\quad \ \ In this section, we discuss several structural properties
of stable-like chains which are crucial for deriving sufficient
conditions for recurrence, transience and ergodicity.

First, recall that, according to \cite[Theorem
10.4.9]{meyn-tweedie-book}, if $\chain{Y}$ is a  recurrent chain,
then it possesses a unique (up to constant multiples) invariant
measure. If the invariant measure is finite, then it may be
normalized to a probability measure. The chain $\chain{Y}$ is called
a \emph{positive chain} if it admits an invariant probability
measure. Otherwise, it is called a \emph{null chain}. In the case of
stable-like chains, from \cite[Theorems 11.5.1 and 11.5.2]{meyn-tweedie-book},
we have the following result.
\begin{theorem}\label{tm2.2}
Let $\alpha:\R\longrightarrow(0,2)$ be an arbitrary function such
that
$$\displaystyle\limsup_{|x|\longrightarrow\infty}\alpha(x)=:\alpha<2$$
and let $\beta\in(0,1]\cap(0,\inf\{\alpha(x):x\in\R\})$ be
arbitrary. Furthermore, let  $\{f_x:x\in\R\}$ be a family of density
functions on $\R$ and $c:\R\longrightarrow(0,\infty)$  which satisfy
conditions (C1)-(C5) and
 such that
 \begin{align*}\liminf_{\delta\longrightarrow0}\liminf_{|x|\longrightarrow\infty}\frac{|x|^{\alpha(x)}}{c(x)}\int_{-\delta|x|}^{\delta|x|}\log\left(1+\mathrm
{sgn}\it\,
(x)\frac{y}{\rm{1}+\it{|x|}}\right)f_{x}(y)dy>R_{\rm{1}}(\alpha)
\end{align*} 
or
\begin{align*}\liminf_{\delta\longrightarrow0}\liminf_{|x|\longrightarrow\infty}\frac{|x|^{\alpha(x)}}{c(x)}\int_{-\delta|x|}^{\delta|x|}\left(\left(1+\mathrm
{sgn}\it\,
(x)\frac{y}{\it{|x|}}\right)^{\beta}-1\right)f_{x}(y)dy>R_2(\alpha,\beta)
\end{align*} 
 holds. Then the stable-like Markov chain
$\{X_n\}_{n\geq0}$     is a null chain, provided it is recurrent.
\end{theorem}
The proof of this statement completely follows the proof of Theorem
\ref{tm1.1}, hence
we omit it here. Recall that
 every
positive chain must be recurrent (see  \cite[Proposition 10.1.1]{meyn-tweedie-book}) and a random walk is never
 positive  since it cannot possess a finite invariant measure (see \cite[Exercise 29.6]{sato-book}).

Next, recall that a Markov chain $\chain{Y}$ is called
\emph{aperiodic} if there does not exist a partition
$\R^{d}=P_1\cup\ldots\cup P_m$ for some $m\geq2$,
$P_1,\ldots,P_m\in\mathcal{B}(\R^{d})$, such that $p(x, P_{i+1}) =
1$ for all $x\in P_i$ and all $1\leq i\leq m - 1$ and $p(x, P_1) =
1$ for all $x\in P_m$. The aperiodicity of stable-like chains has
been shown in  \cite[Proposition 2.4]{sandric-rectrans}. Now, as in
the countable state case, one would expect that every positive
recurrent and aperiodic chain is automatically ergodic, but in
general this is not true. The recurrence property is too weak. We
need  Harris recurrence  (see \cite[Theorem
13.2.5]{meyn-tweedie-book}). Hence, every positive Harris recurrent
and aperiodic chain is ergodic (see \cite[Theorem
13.3.3]{meyn-tweedie-book}). In the case of stable-like chains we
have even more.  According to \cite[Proposition
10.1.1]{meyn-tweedie-book}, in the case of stable-like chains these
two properties coincide. Thus, in order to check ergodicity of
stable-like chains it suffices to check that the corresponding
invariant measure is finite, and, by \cite[Theorem
13.0.1]{meyn-tweedie-book},  sufficient conditions for the
finiteness of the corresponding invariant measure have been given in
Theorem \ref{tm1.3}.

Further, as a simple consequence of Theorems \ref{tm1.1} and
\ref{tm1.3}, \cite[Theorems 9.2.2, 17.0.1 and
18.3.2]{meyn-tweedie-book} and \cite[Proposition
5.2]{sandric-rectrans}, we get the following additional long-time
properties of stable-like chains.
\begin{corollary}\begin{enumerate}
\item [(i)] Under   assumptions of Theorem  \ref{tm1.1},  for every initial position
$x\in\R$ and every   covering $\{O_n\}_{n\in\N}$ of $\R$ by open
bounded sets we have
$$\mathbb{P}^{x}\left(\bigcap_{n=1}^{\infty}\left\{\sum_{k=1}^{\infty}1_{\{X_k\in
O_n\}}<\infty\right\}\right)=0.$$ In other words, for every initial
position $x\in\R$ the event $\{X_n\in C^{c}\ \textrm{for any compact
set}\ C\subseteq\R\ \textrm{and all}\ n\in\N\ \textrm{sufficiently
large}\}$ has probability $0$.
\item[(ii)] Under   assumptions of Theorem  \ref{tm1.3}, for every initial position
$x\in\R$ and every $\varepsilon>0$ there exists a compact set
$C\subseteq\R$, such that
$$\liminf_{n\longrightarrow\infty}\mathbb{P}^{x}(X_n\in
C)\geq1-\varepsilon\quad\textrm{and}\quad\liminf_{
n\longrightarrow\infty}\mathbb{E}^{x}\left[\frac{1}{n}\sum_{k=1}^{n}1_{\{X_k\in
C\}}ds\right]\geq 1-\varepsilon.$$
\end{enumerate}
\end{corollary}

We end this section with the following observation. Assume that
$\chain{Y}$ is an ergodic Markov chain with invariant measure
$\pi(\cdot)$. Then, clearly,
$$\lim_{n\longrightarrow\infty}\mathbb{E}^{x}[f(Y_n)]=\int_{\R^{d}}f(y)\pi(dy)=:\pi(f)$$
holds for all $x\in\R^{d}$ and all bounded Borel measurable
functions $f(y).$
 In what follows,  we extend this convergence to a wider class of functions. For any Borel measurable function
$f(y)\geq 1$ and any signed measure $\mu(\cdot)$ on
$\mathcal{B}(\R^{d})$ we write
$$||\mu||_f:=\sup_{|g|\leq f}|\mu(g)|.$$ A  Markov chain $\chain{Y}$ is called $f$-\emph{ergodic} if it is positive Harris
recurrent with the invariant probability measure $\pi(\cdot)$, if
$\pi(f) < \infty$, and if
$$\lim_{n\longrightarrow\infty}||p^{n}(x,\cdot)-\pi(\cdot)||_f=0$$ holds for all  $x\in\R^{d}.$
Note that $||\cdot||_1=||\cdot||$. Hence,   $f$-ergodicity implies
ergodicity. Now, from \cite[Theorem 14.0.1]{meyn-tweedie-book}, we
have the following.
\begin{theorem}\label{tm2.4} Let $\alpha:\R\longrightarrow(0,2)$ be an arbitrary function such
that
$$0<\inf\{\alpha(x):x\in\R\},$$ let
$\alpha:=\liminf_{|x|\longrightarrow\infty}\alpha(x)$ and let
$\beta\in(0,1]\cap(0,\inf\{\alpha(x):x\in\R\})$ be arbitrary.
Furthermore, let $\{f_x:x\in\R\}$ be a family of density functions
on $\R$ and $c:\R\longrightarrow(0,\infty)$  which satisfy
conditions (C1)-(C5) and
 such that
\begin{align}\label{eq:2.4}\limsup_{d\longrightarrow0}\limsup_{\delta\longrightarrow0}\limsup_{|x|\longrightarrow\infty}\frac{|x|^{\alpha(x)}}{c(x)}\left(\int_{-\delta|x|}^{\delta|x|}\log\left(1+\mathrm
{sgn}\it\,
(x)\frac{y}{\rm{1}+\it{|x|}}\right)f_{x}(y)dy+dg(x)\right)<R_1(\alpha)
\end{align}
or
\begin{align}\label{eq:2.5}\limsup_{d\longrightarrow0}\limsup_{\delta\longrightarrow0}\limsup_{|x|\longrightarrow\infty}\frac{|x|^{\alpha(x)}}{c(x)}\left(\int_{-\delta|x|}^{\delta|x|}\left(\left(1+\mathrm
{sgn}\it\,
(x)\frac{y}{\it{|x|}}\right)^{\beta}-1\right)f_{x}(y)dy+d\left(g(x)\right)^{-\beta}\right)<R_2(\alpha,\beta),
\end{align} for some Borel measurable function $g(x)\geq1$.
Then the stable-like Markov chain $\{X_n\}_{n\geq0}$     is
$f$-ergodic for every Borel measurable function $f(x)\geq1$ such
that $f(x)\leq g(x)$.
\end{theorem}

\section{Proof of  the main results}
\quad \ \ In this section we give the proofs of Theorems
\ref{tm1.1}, \ref{tm1.2}, \ref{tm1.3} and \ref{tm2.4} and Corollary
\ref{c1.4}. Before the proofs, we recall several special functions
we need. The Gamma function is defined by the formula
$$\Gamma(z):=\int_0^{\infty}t^{z-1}e^{-t}dt,\quad z\in\CC,\ \mathrm{Re}\it(z)>\rm
0,$$ and  it can be analytically continued on $\CC\setminus\ZZ_-$.
The Digamma function is a function defined by
$\Psi(z):=\Gamma'(z)/\Gamma(z),$ for $z\in \CC\setminus\ZZ_-,$ and
it satisfies the following properties:
\begin{enumerate}
\item [(i)]
  \be\label{eq:3.1}\Psi(1+z)=-\gamma+\sum_{n=1}^{\infty}\frac{z}{n(n+z)},\ee
  where $\gamma$ is Euler's number;
\item [(ii)]\be\label{eq:3.2}\Psi(1+z)=\Psi(z)+\frac{1}{z};\ee
  \item [(iii)]
\be\label{eq:3.3}\Psi(2z)=\frac{1}{2}\Psi(z)+\frac{1}{2}\Psi\left(z+\frac{1}{2}\right)+\log2;\ee
  \item [(iv)]
\be\label{eq:3.4}\Psi(1-z)=\Psi(z)+\pi\rm{ctg}\,(\pi z).\ee
\end{enumerate}

The Gauss hypergeometric function  is defined by the formula
\be\label{eq:3.5}_2F_1(a,b,c;z):=\sum_{n=0}^{\infty}\frac{(a)_n(b)_n}{(c)_n}\frac{z^{n}}{n!},\ee
for $a,b,c,z\in \CC$, $c\notin\ZZ_-$, where for $w\in\CC$ and
$n\in\ZZ_+$, $(w)_n$ is defined by
$$(w)_0=1\quad \textrm{and}\quad (w)_n=w(w+1)\cdots(w+n-1).$$
The series (\ref{eq:3.5}) absolutely converges on $|z|<1$,
absolutely converges on $|z|\leq1$ when $\mathrm{Re}\it(c-a-b)>\rm
0$, conditionally converges on $|z|\leq1$, except for $z=1$, when
$-1<\mathrm{Re} \it(c-b-a)\leq \rm0$ and diverges when
$\mathrm{Re}\it (c-b-a)\leq\rm-1$. In the case when $\mathrm{Re}
\it(c)>\mathrm{Re} \it (b)>\rm0$, it can be analytically continued
on $\CC\setminus(1,\infty)$ by the formula
\be\label{eq:3.6}_2F_1(a,b,c;z)=\frac{\Gamma(c)}{\Gamma(b)\Gamma(c-b)}\int_{0}^{1}t^{b-1}(1-t)^{c-b-1}(1-tz)^{-a}dt.\ee
For further properties of the Gamma function, Digamma function and
hypergeometric functions see \cite[Chapters 6 and
15]{abram-stegun-book}.

\begin{proof}[Proof of Theorem \ref{tm1.1}]
In \cite[Theorem 1.3]{sandric-rectrans} it has been proved that if
$1<\alpha:=\liminf_{|x|\longrightarrow\infty}\alpha(x)$ and if
\begin{align}\label{eq:3.7}&\limsup_{\delta\longrightarrow0}\limsup_{|x|\longrightarrow\infty}\frac{|x|^{\alpha(x)}}{c(x)}\int_{-\delta|x|}^{\delta|x|}\log\left(1+\mathrm
{sgn}\it\,
(x)\frac{y}{\rm{1}+\it{|x|}}\right)f_{x}(y)dy\nonumber\\&<\displaystyle\sum_{i=1}^{\infty}\frac{1}{i(2i-\alpha)}-\frac{\log
2}{\alpha}-\frac{1}{2\alpha}\left(\Psi\left(\frac{\alpha+1}{2}\right)-\Psi\left(\frac{\alpha}{2}\right)\right)+\frac{\gamma}{\alpha}+\frac{\Psi(\alpha)}{\alpha},
\end{align} then the corresponding stable-like chain is recurrent.
Further, by assuming (\ref{eq:3.7}) and performing completely the
same proof as in \cite[Theorem 1.3]{sandric-rectrans},  the same
conclusion also remains true  in the case when
$0<\alpha:=\liminf_{|x|\longrightarrow\infty}\alpha(x).$ Finally,
from (\ref{eq:3.1}),  (\ref{eq:3.3}) and (\ref{eq:3.4}), the
right-hand side in  (\ref{eq:3.7}) equals $R_1(\alpha)$, that is,
(\ref{eq:3.7}) becomes (\ref{eq:1.2}). Thus, the first claim
follows.

To prove the second claim, we divide the proof into four steps.

\textbf{Step 1.} In the first step we explain  our strategy of the
proof. Let
$\beta\in(0,1]\cap(0,\liminf_{|x|\longrightarrow\infty}\alpha(x))$
be arbitrary and let us define the function
$V:\R\longrightarrow[0,\infty)$ by the formula
$$V(x):=|x|^{\beta}.$$
Clearly, the level set $C_V(r):=\{y:V(y)\leq r\}$ is a bounded set
for every $r>0$. Thus, by \cite[Prposition 2.6]{sandric-rectrans}
and \cite[Theorem 8.4.3]{meyn-tweedie-book}, it suffices to show
that there exists $r_0>0$ such that
$$\int_{\R}p(x,dy)V(y)-V(x)\leq0$$ for all $x\in C^{c}_V(r_0)$. Next,
since $C_V(r)\uparrow\R$, for $r\longrightarrow\infty$, it is enough
to show that
$$\displaystyle\limsup_{|x|\longrightarrow\infty}\frac{|x|^{\alpha(x) -\beta}}{c(x)}\left(\int_{\R}p(x,dy)V(y)-V(x)\right)<0.$$
We have
\begin{align}&|x|^{-\beta}\left(\displaystyle\int_{\R}p(x,dy)V(y)-V(x)\right)=|x|^{-\beta}\left(\displaystyle\int_{\R}f_x(y)V(y+x)dy-V(x)\right)\nonumber\\
=&\label{eq:3.8}\displaystyle\int_{\{y+x>0\}}\left(\left(\frac{x+y}{|x|}\right)^{\beta}-1\right)f_x(y)dy+\displaystyle\int_{\{y+x<0\}}\left(\left(\frac{-x-y}{|x|}\right)^{\beta}-1\right)f_x(y)dy.\end{align}

\textbf{Step 2.} In the second step we  find an appropriate upper
bound for the first summand in (\ref{eq:3.8}).   For any $x>0$ we
have
\begin{align*}\displaystyle\int_{\{y+x>0\}}\left(\left(\frac{x+y}{x}\right)^{\beta}-1\right)f_x(y)dy
=\displaystyle\int_{\{y+x>0\}}\left(\left(1+\frac{y}{x}\right)^{\beta}-1\right)f_x(y)dy.
\end{align*}
 Let $0<\delta<1$ be
arbitrary. By restricting the function $(1+t)^{\beta}-1$ to the
intervals $(-1,-\delta),$ $[-\delta,\delta]$, $(\delta,1)$ and
$[1,\infty)$, and using its  Binomial series, that is,
$$(1+t)^{\beta}-1=\displaystyle\sum_{i=1}^{\infty}{\beta \choose i} t^{i},$$
for $t\in(-1,1)$, where $${\beta \choose
i}=\frac{\beta(\beta-1)\cdots(\beta-i+1)}{i!},$$
 from  Fubini's theorem we get
\begin{align*}&\displaystyle\int_{\{y+x>0\}}\left(\left(\frac{x+y}{x}\right)^{\beta}-1\right)f_x(y)dy\\
&=\displaystyle\sum_{i=1}^{\infty}{\beta \choose
i}\frac{1}{x^{i}}\displaystyle\int_{\{-x<y<-\delta x\}\cap
\{y+x>0\}}y^{i}f_x(y)dy+\displaystyle\int_{\{-\delta x\leq
y\leq\delta x\}\cap
\{y+x>0\}}\left(\left(1+\frac{y}{x}\right)^{\beta}-1\right)f_x(y)dy\\
&\ \ \ +\displaystyle\sum_{i=1}^{\infty}{\beta \choose
i}\frac{1}{x^{i}}\displaystyle\int_{\{\delta x<y<x\}\cap
\{y+x>0\}}y^{i}f_x(y)dy+\displaystyle\int_{\{y\geq x\}\cap
\{y+x>0\}}\left(\left(1+\frac{y}{x}\right)^{\beta}-1\right)f_x(y)dy.\end{align*}
Further,   we have
\begin{align*}&\displaystyle\int_{\{y+x>0\}}\left(\left(\frac{x+y}{x}\right)^{\beta}-1\right)f_x(y)dy\\
&=\displaystyle\sum_{i=1}^{\infty}{\beta \choose
i}\frac{1}{x^{i}}\displaystyle\int_{\{-x<y<-\delta x\}}y^{i}f_x(y)dy+\displaystyle\int_{\{-\delta x\leq y\leq\delta x\}}\left(\left(1+\frac{y}{x}\right)^{\beta}-1\right)f_x(y)dy\\
&\ \ \ +\displaystyle\sum_{i=1}^{\infty}{\beta \choose
i}\frac{1}{x^{i}}\displaystyle\int_{\{\delta x<y<x\}}y^{i}f_x(y)dy
+\displaystyle\int_{\{y\geq
x\}}\left(\left(1+\frac{y}{x}\right)^{\beta}-1\right)f_x(y)dy.\end{align*}
Let us put
\begin{align*}
U^{\delta}_1(x):=&-\frac{\beta}{x}\displaystyle\int_{\{\delta x<y<x\}}yf_x(-y)dy+\frac{\beta}{x}\displaystyle\int_{\{\delta x<y<x\}}yf_x(y)dy\\
U^{\delta}_2(x):=&\displaystyle\sum_{i=2}^{\infty}{\beta \choose
i}\frac{(-1)^{i}}{x^{i}}\displaystyle\int_{\{\delta
x<y<x\}}y^{i}f_x(-y)dy+\displaystyle\sum_{i=2}^{\infty}{\beta
\choose
i}\frac{1}{x^{i}}\displaystyle\int_{\{\delta x<y<x\}}y^{i}f_x(y)dy,\\
U^{\delta}_3(x):=&\int_{\{-\delta x\leq y\leq\delta
x\}}\left(\left(1+\frac{y}{x}\right)^{\beta}-1\right)f_x(y)dy
\quad \textrm{and}\\
U_4(x):=&\displaystyle\int_{\{y\geq
x\}}\left(\left(1+\frac{y}{x}\right)^{\beta}-1\right)f_x(y)dy,\end{align*}
for $0<\delta<1$.  Hence, we find
\be\label{eq:3.9}\int_{\{y+x>0\}}\left(\left(\frac{x+y}{x}\right)^{\beta}-1\right)f_x(y)dy=
U^{\delta}_1(x)+ U^{\delta}_2(x)+U^{\delta}_3(x)+U_4(x).\ee Here
comes the crucial step where   condition (C3) is needed. In the
above terms, by  (C3), we can replace all the density functions
$f_x(y)$ by the functions $c(x)|y|^{-\alpha(x)-1}$  and find a more
operable upper bound in (\ref{eq:3.9}). Let $0<\varepsilon<1$ be
arbitrary. Then, by  (C3), there exists $y_{\varepsilon}\geq 1$,
such that for all $|y|\geq y_{\varepsilon}$
$$\left|f_x(y)\frac{|y|^{\alpha(x)+1}}{c(x)}-1\right|<\varepsilon,$$ for all
$x\in[-k_0,k_0]^{c}$ (recall that the constant $k_0$ is defined in
condition (C3)). Let $x>\left(k_0\vee
y_{\varepsilon}/\delta\right).$
 By  a simple computation,  we have
\begin{align*}U^{\delta}_1(x)<&\frac{2\varepsilon c(x)\beta}{(\alpha(x)-1)x^{\alpha(x)}}\left(\delta^{-\alpha(x)+1}-1\right),
\end{align*}in the case when $\alpha(x)\neq1$, and \begin{align*}U^{\delta}_1(x)<&\frac{2\varepsilon c(x)\beta}{x}\log\left(\frac{1}{\delta}\right),
\end{align*} in the case when $\alpha(x)=1.$ Let us denote the right
hand side in the above inequalities by
$U_1^{\delta,\varepsilon}(x).$ Further, we have
\begin{align*}
U_2^{\delta}(x)<&\frac{c(x)}{x^{\alpha(x)}}\displaystyle\sum_{i=2}^{\infty}{\beta
\choose
i}\frac{1+(-1)^{i}-2\varepsilon(-1)^{i}}{i-\alpha(x)}\left(1-\delta^{i-\alpha(x)}\right)=:U_2^{\delta,\varepsilon}(x)\quad \textrm{and}\\
U_4(x)<&(1+\varepsilon)c(x)\int_{x}^{\infty}\left(\left(1+\frac{y}{x}\right)^{\beta}-1\right)\frac{dy}{y^{\alpha(x)+1}}=:U^{\varepsilon}_4(x).\end{align*}
Hence, from (\ref{eq:3.9}), we get
\be\label{eq:3.10}\int_{\{y+x>0\}}\left(\left(\frac{x+y}{x}\right)^{\beta}-1\right)f_x(y)dy<U^{\delta,\varepsilon}_1(x)+
U^{\delta,\varepsilon}_2(x)+U^{\delta}_3(x)+U^{\varepsilon}_4(x).\ee

\textbf{Step 3.} In the third step we find an appropriate upper
bound for the second summand in (\ref{eq:3.8}).  Let
$x>\left(k_0\vee y_{\varepsilon}/\delta\right).$ Then, again by
(C3),
\begin{align*}&\displaystyle\int_{\{y+x<0\}}\left(\left(\frac{-x-y}{x}\right)^{\beta}-1\right)f_x(y)dy\\&<c(x)(1-\varepsilon)\int_{x}^{2x}\left(\left(-1+\frac{y}{x}\right)^{\beta}-1\right)\frac{dy}{y^{\alpha(x)+1}}+c(x)(1+\varepsilon)\int_{2x}^{\infty}\left(\left(-1+\frac{y}{x}\right)^{\beta}-1\right)\frac{dy}{y^{\alpha(x)+1}}\\
&=c(x)(1-\varepsilon)\int_{x}^{\infty}\left(\left(-1+\frac{y}{x}\right)^{\beta}-1\right)\frac{dy}{y^{\alpha(x)+1}}+2\varepsilon
c(x)\int_{2x}^{\infty}\left(\left(-1+\frac{y}{x}\right)^{\beta}-1\right)\frac{dy}{y^{\alpha(x)+1}}.\end{align*}
Note that in the first inequality we make a change of variables
 $y\longmapsto-y$. Let us put
\begin{align*}U_5^{\varepsilon}(x):=&c(x)(1-\varepsilon)\int_{x}^{\infty}\left(\left(-1+\frac{y}{x}\right)^{\beta}-1\right)\frac{dy}{y^{\alpha(x)+1}}\\&+2\varepsilon
c(x)\int_{2x}^{\infty}\left(\left(-1+\frac{y}{x}\right)^{\beta}-1\right)\frac{dy}{y^{\alpha(x)+1}}.\end{align*}
We have
\be\label{eq:3.11}\displaystyle\int_{\{y+x<0\}}\left(\left(-1+\frac{y}{x}\right)^{\beta}-1\right)f_x(y)dy
<U_5^{\varepsilon}(x).\ee

\textbf{Step 4.} In the fourth step we prove
$$\limsup_{x\longrightarrow\infty}\frac{x^{\alpha(x)-\beta}}{c(x)}\left(\int_{\R}p(x,dy)V(y)-V(x)\right)<0.$$
 By combining
(\ref{eq:3.8}), (\ref{eq:3.9}), (\ref{eq:3.10}) and (\ref{eq:3.11})
we have
\be\label{eq:3.12}x^{-\beta}\left(\int_{\R}p(x,dy)V(y)-V(x)\right)<U^{\delta,\varepsilon}_1(x)+
U^{\delta,\varepsilon}_2(x)+U^{\delta}_3(x)+U^{\varepsilon}_4(x)+U_5^{\varepsilon}(x).\ee
In the rest of the fourth step we prove
\begin{eqnarray*}&&\limsup_{x\longrightarrow\infty}\frac{x^{\alpha(x)-\beta}}{c(x)}\left(\int_{\R}p(x,dy)V(y)-V(x)\right)\\ &&\begin{split}<&\limsup_{\delta\longrightarrow0}\limsup_{\varepsilon\longrightarrow0}\limsup_{x\longrightarrow\infty}\frac{x^{\alpha(x)}}{c(x)}U^{\delta,\varepsilon}_1(x)+\limsup_{\delta\longrightarrow0}\limsup_{\varepsilon\longrightarrow0}\limsup_{x\longrightarrow\infty}\frac{x^{\alpha(x)}}{c(x)}U^{\delta,\varepsilon}_2(x)\\&+\limsup_{\varepsilon\longrightarrow0}\limsup_{x\longrightarrow\infty}\frac{x^{\alpha(x)}}{c(x)}(U_4^{\varepsilon}(x)+U_5^{\varepsilon}(x))+R_2(\alpha,\beta)\leq0.\end{split}\end{eqnarray*}
 Recall
that $0<\alpha=\liminf_{|x|\longrightarrow\infty}\alpha(x)$,
$$R_2(\alpha,\beta)=-\sum_{n=1}^{\infty}{\beta \choose
2n}\frac{2}{2n-\alpha}+\frac{2}{\alpha}-\frac{_2F_1(-\beta,\alpha-\beta,1+\alpha-\beta;-1)+\
_2F_1(-\beta,\alpha-\beta,1+\alpha-\beta;1)}{\alpha-\beta}$$ and
$$\limsup_{\delta\longrightarrow0}\limsup_{x\longrightarrow\infty}\frac{x^{\alpha(x)}}{c(x)}U_3^{\delta}(x)<R_2(\alpha,\beta)$$
(assumption (\ref{eq:1.3})). Clearly,
\begin{align}\label{eq:3.13}&\limsup_{\delta\longrightarrow0}\limsup_{\varepsilon\longrightarrow0}\limsup_{x\longrightarrow\infty}\frac{x^{\alpha(x)}}{c(x)}U^{\delta,\varepsilon}_1(x)=0.\end{align}
  Further, by  the
dominated convergence theorem, we have
\begin{eqnarray}&&\label{eq:3.14}\limsup_{\delta\longrightarrow0}\limsup_{\varepsilon\longrightarrow0}\limsup_{x\longrightarrow\infty}\frac{x^{\alpha(x)}}{c(x)}U^{\delta,\varepsilon}_2(x)\nonumber\\&&\begin{split}=\limsup_{\delta\longrightarrow0}\limsup_{\varepsilon\longrightarrow0}\limsup_{x\longrightarrow\infty}&\displaystyle\sum_{i=2}^{\infty}{\beta
\choose
i}\frac{1+(-1)^{i}-2\varepsilon(-1)^{i}}{i-\alpha(x)}\left(1-\delta^{i-\alpha(x)}\right)\end{split}\nonumber\\
&&\begin{split}=\limsup_{x\longrightarrow\infty}&\displaystyle\sum_{i=1}^{\infty}{\beta
\choose
2i}\frac{2}{2i-\alpha(x)}\end{split}\nonumber\\
&&=\displaystyle\sum_{i=1}^{\infty}{\beta \choose
2i}\frac{2}{2i-\alpha}.\end{eqnarray} Now, let us compute
$$\limsup_{\varepsilon\longrightarrow0}\limsup_{x\longrightarrow\infty}\frac{x^{\alpha(x)}}{c(x)}(U^{\varepsilon}_4(x)+U^{\varepsilon}_5(x)).$$
By (\ref{eq:3.6}), we have
\begin{align}\label{eq:3.15}&\limsup_{\varepsilon\longrightarrow0}\limsup_{x\longrightarrow\infty}\frac{x^{\alpha(x)}}{c(x)}(U^{\varepsilon}_4(x)+U^{\varepsilon}_5(x))\nonumber\\
&=\limsup_{\varepsilon\longrightarrow0}\limsup_{x\longrightarrow\infty}\Bigg[(1+\varepsilon)\left(-\frac{1}{\alpha(x)}+\frac{
_2F_1(-\beta,\alpha(x)-\beta,1+\alpha(x)-\beta;-1)}{\alpha(x)-\beta}\right)\nonumber\\
&\ \ \ +(1-\varepsilon)\left(-\frac{1}{\alpha(x)}+\frac{
_2F_1\left(-\beta,\alpha(x)-\beta,1+\alpha(x)-\beta;1\right)}{\alpha(x)-\beta}\right)\nonumber\\
&\ \ \  +2\varepsilon \left(-\frac{1}{\alpha(x)2^{\alpha(x)}}-\frac{
_2F_1\left(-\beta,\alpha(x)-\beta,1+\alpha(x)-\beta;\frac{1}{2}\right)}{(\beta-\alpha(x))2^{\beta-\alpha(x)}}\right)\Bigg]\nonumber\\
&=\limsup_{x\longrightarrow\infty}\left(-\frac{2}{\alpha(x)}+\frac{_2F_1(-\beta,\alpha(x)-\beta,1+\alpha(x)-\beta;-1)+\ _2F_1(-\beta,\alpha(x)-\beta,1+\alpha(x)-\beta;1)}{\alpha(x)-\beta}\right)\nonumber\\
&=-\frac{2}{\alpha}+\frac{_2F_1(-\beta,\alpha-\beta,1+\alpha-\beta;-1)+\
_2F_1(-\beta,\alpha-\beta,1+\alpha-\beta;1)}{\alpha-\beta},\end{align}
where in the second equality we use (\ref{eq:3.5}) and the fact that
all terms are bounded, and in the last equality we use the fact that
the function
$$x\longmapsto-\frac{2}{x}+\frac{_2F_1(-\beta,x-\beta,1+x-\beta;-1)+\
_2F_1(-\beta,x-\beta,1+x-\beta;1)}{x-\beta}$$ is decreasing on
$(\beta,2)$. Now, by combining (\ref{eq:1.3}), (\ref{eq:3.12}),
(\ref{eq:3.13}), (\ref{eq:3.14}) and (\ref{eq:3.15}) we have
$$\limsup_{x\longrightarrow\infty}\frac{x^{\alpha(x)-\beta}}{c(x)}\left(\int_{\R}p(x,dy)V(y)-V(x)\right)<0.$$
The case when $x<0$ is treated in the same way.  Therefore,
 we have proved the desired result.
\end{proof}

\begin{proof}[Proof of Theorem \ref{tm1.2}]We proceed similarly as in the proof of Theorem \ref{tm1.1}.  The proof is divided  into four steps.

\textbf{Step 1.} In the first step we explain  our strategy of the
proof. Let $\beta\in(0,1)$ be arbitrary and let us define the
function $V:\R\longrightarrow[0,\infty)$ by the formula
$$V(x):=1-(1+|x|)^{-\beta}.$$
Clearly, the sets $C_V(r):=\{y:V(y)\leq r\}$ and $C^{c}_V(r)$ have
positive Lebesgue measure for every $0 < r < 1$. Thus, by
\cite[Theorem 8.4.2]{meyn-tweedie-book}, it suffices to show that
there exists $0 < r_0 < 1$ such that
$$\int_{\R}p(x,dy)V(y)-V(x)\geq0$$ for all $x\in C^{c}_V(r_0)$. Next,
since $C_V(r)\uparrow\R$ for $r\longrightarrow1$, it is enough to
show that
$$\displaystyle\liminf_{|x|\longrightarrow\infty}\frac{(1+|x|)^{\alpha(x) +\beta}}{c(x)}\left(\int_{\R}p(x,dy)V(y)-V(x)\right)>0.$$
We have
\begin{align}&(1+|x|)^{\beta}\left(\displaystyle\int_{\R}p(x,dy)V(y)-V(x)\right)=(1+|x|)^{\beta}\left(\displaystyle\int_{\R}f_x(y)V(y+x)dy-V(x)\right)\nonumber\\
=&\label{eq:3.16}\displaystyle\int_{\{y+x>0\}}\left(1-\left(\frac{1+x+y}{1+|x|}\right)^{-\beta}\right)f_x(y)dy+\displaystyle\int_{\{y+x<0\}}\left(1-\left(\frac{1-x-y}{1+|x|}\right)^{-\beta}\right)f_x(y)dy.\end{align}

\textbf{Step 2.} In the second step we  find an appropriate lower
bound for the first summand in (\ref{eq:3.16}).   For any $x>0$ we
have
\begin{align*}\displaystyle\int_{\{y+x>0\}}\left(1-\left(\frac{1+x+y}{1+x}\right)^{-\beta}\right)f_x(y)dy
=\displaystyle\int_{\{y+x>0\}}\left(1-\left(1+\frac{y}{1+x}\right)^{-\beta}\right)f_x(y)dy.
\end{align*}
 Let $0<\delta<1$ be
arbitrary. By restricting the function $1-(1+t)^{-\beta}$ to the
intervals $(-1,-\delta),$ $[-\delta,\delta]$, $(\delta,1)$ and
$[1,\infty)$, and using its Binomial series, that is,
$$1-(1+t)^{-\beta}=-\displaystyle\sum_{i=1}^{\infty}{-\beta \choose i} t^{i},$$
for $t\in(-1,1)$, from  Fubini's theorem we get
\begin{align*}&\displaystyle\int_{\{y+x>0\}}\left(1-\left(\frac{1+x+y}{1+x}\right)^{-\beta}\right)f_x(y)dy\\
&=-\displaystyle\sum_{i=1}^{\infty}{-\beta \choose
i}\frac{1}{(1+x)^{i}}\displaystyle\int_{\{-1-x<y<-\delta(1+x)\}\cap
\{y+x>0\}}y^{i}f_x(y)dy\\
&\ \ \ +\displaystyle\int_{\{-\delta(1+x)\leq y\leq\delta(1+x)\}\cap
\{y+x>0\}}\left(1-\left(1+\frac{y}{1+x}\right)^{-\beta}\right)f_x(y)dy\\
&\ \ \ -\displaystyle\sum_{i=1}^{\infty}{-\beta \choose
i}\frac{1}{(1+x)^{i}}\displaystyle\int_{\{\delta(1+x)<y<1+x\}\cap
\{y+x>0\}}y^{i}f_x(y)dy\\
&\ \ \ +\displaystyle\int_{\{y\geq1+x\}\cap
\{y+x>0\}}\left(1-\left(1+\frac{y}{1+x}\right)^{-\beta}\right)f_x(y)dy.\end{align*}
Furthermore, by taking $x>\delta/(1-\delta)$  we get
\begin{align*}&\displaystyle\int_{\{y+x>0\}}\left(1-\left(\frac{1+x+y}{1+x}\right)^{-\beta}\right)f_x(y)dy\\
&=-\displaystyle\sum_{i=1}^{\infty}{-\beta \choose
i}\frac{1}{(1+x)^{i}}\displaystyle\int_{\{-x<y<-\delta(1+x)\}}y^{i}f_x(y)dy\\
&\ \ \ +\displaystyle\int_{\{-\delta(1+x)\leq y\leq\delta(1+x)\}}\left(1-\left(1+\frac{y}{1+x}\right)^{-\beta}\right)f_x(y)dy\\
&\ \ \ -\displaystyle\sum_{i=1}^{\infty}{-\beta \choose
i}\frac{1}{(1+x)^{i}}\displaystyle\int_{\{\delta(1+x)<y<1+x\}}y^{i}f_x(y)dy\\
&\ \ \
+\displaystyle\int_{\{y\geq1+x\}}\left(1-\left(1+\frac{y}{1+x}\right)^{-\beta}\right)f_x(y)dy.\end{align*}
Let us put
\begin{align*}
U^{\delta}_1(x):=&-\frac{\beta}{1+x}\displaystyle\int_{\{\delta(1+x)<y<x\}}yf_x(-y)dy+\frac{\beta}{1+x}\displaystyle\int_{\{\delta(1+x)<y<1+x\}}yf_x(y)dy\\
U^{\delta}_2(x):=&-\displaystyle\sum_{i=2}^{\infty}{-\beta \choose
i}\frac{(-1)^{i}}{(1+x)^{i}}\displaystyle\int_{\{\delta(1+x)<y<x\}}y^{i}f_x(-y)dy\\&-\displaystyle\sum_{i=2}^{\infty}{-\beta
\choose
i}\frac{1}{(1+x)^{i}}\displaystyle\int_{\{\delta(1+x)<y<1+x\}}y^{i}f_x(y)dy,\\
U^{\delta}_3(x):=&\int_{\{-\delta(1+x)\leq
y\leq\delta(1+x)\}}\left(1-\left(1+\frac{y}{1+x}\right)^{-\beta}\right)f_x(y)dy
\quad \textrm{and}\\
U_4(x):=&\displaystyle\int_{\{y\geq1+x\}}\left(1-\left(1+\frac{y}{1+x}\right)^{-\beta}\right)f_x(y)dy,\end{align*}
for $0<\delta<1$ and $x>\delta/(1-\delta)$.  Hence, we find
\be\label{eq:3.17}\int_{\{y+x>0\}}\left(1-\left(\frac{1+x+y}{1+x}\right)^{-\beta}\right)f_x(y)dy=
U^{\delta}_1(x)+ U^{\delta}_2(x)+U^{\delta}_3(x)+U_4(x).\ee Now, we
apply (C3) and find  a more operable lower bound in (\ref{eq:3.17}).
Let $0<\varepsilon<1$ be arbitrary. Then, by (C3), there exists
$y_{\varepsilon}\geq 1$, such that for all $|y|\geq y_{\varepsilon}$
$$\left|f_x(y)\frac{|y|^{\alpha(x)+1}}{c(x)}-1\right|<\varepsilon,$$ for all
$x\in[-k_0,k_0]^{c}$ (recall that the constant $k_0$ is defined in
condition (C3)).
 Let
$x>\left(k_0\vee(y_{\varepsilon}-\delta)/\delta\vee\delta/(1-\delta)\right).$
 By  a simple computation,  we have
\begin{align*}U^{\delta}_1(x)>&-\frac{(1+\varepsilon)c(x)\beta}{(\alpha(x)-1)(1+x)^{\alpha(x)}}\left(\delta^{-\alpha(x)+1}-\left(\frac{x}{1+x}\right)^{-\alpha(x)+1}\right)\\&+\frac{(1-\varepsilon)c(x)\beta}{(\alpha(x)-1)(1+x)^{\alpha(x)}}\frac{\delta-\delta^{\alpha(x)}}{\delta^{\alpha(x)}},
\end{align*}in the case when $\alpha(x)\neq1$, and \begin{align*}U^{\delta}_1(x)>&-\frac{(1+\varepsilon)c(x)\beta}{1+x}\log\left(\frac{x}{\delta(1+x)}\right)+\frac{(1-\varepsilon)c(x)\beta}{1+x}\log\left(\frac{1}{\delta}\right),
\end{align*} in the case when $\alpha(x)=1.$ Let us denote the right
hand side in the above inequalities by
$U_1^{\delta,\varepsilon}(x).$ Further, we have
\begin{align*}
U_2^{\delta}(x)>&-\frac{(1+\varepsilon)c(x)}{(1+x)^{\alpha(x)}}\displaystyle\sum_{i=2}^{\infty}{-\beta
\choose
i}\frac{(-1)^{i}}{i-\alpha(x)}\left(\left(\frac{x}{1+x}\right)^{i-\alpha(x)}-\delta^{i-\alpha(x)}\right)\\
&-\frac{c(x)}{(1+x)^{\alpha(x)}}\displaystyle\sum_{i=2}^{\infty}\left({-\beta
\choose
i}\frac{1+(-1)^{i}\varepsilon}{i-\alpha(x)}\frac{\delta^{\alpha(x)}-\delta^{i}}{\delta^{\alpha(x)}}\right)=:U_2^{\delta,\varepsilon}(x)\quad \textrm{and}\\
U_4(x)>&(1-\varepsilon)c(x)\int_{1+x}^{\infty}\left(1-\left(1+\frac{y}{1+x}\right)^{-\beta}\right)\frac{1}{y^{\alpha(x)+1}}dy=:U^{\varepsilon}_4(x).\end{align*}
Hence, from (\ref{eq:3.17}), we get
\be\label{eq:3.18}\int_{\{y+x>0\}}\left(1-\left(\frac{1+x+y}{1+x}\right)^{-\beta}\right)f_x(y)dy>U^{\delta,\varepsilon}_1(x)+
U^{\delta,\varepsilon}_2(x)+U^{\delta}_3(x)+U^{\varepsilon}_4(x).\ee

\textbf{Step 3.} In the third step we find an appropriate lower
bound for the second summand in (\ref{eq:3.16}).  Let
$x>\left(k_0\vee(y_{\varepsilon}-\delta)/\delta\vee\delta/(1-\delta)\right).$
Then,  again by
 (C3), we have
\begin{align*}&\displaystyle\int_{\{y+x<0\}}\left(1-\left(\frac{1-x-y}{1+x}\right)^{-\beta}\right)f_x(y)dy\\&>c(x)(1+\varepsilon)\int_{x}^{2x}\left(1-\left(\frac{1-x+y}{1+x}\right)^{-\beta}\right)\frac{1}{|y|^{\alpha(x)+1}}dy\\&\ \ \ +c(x)(1-\varepsilon)\int_{2x}^{\infty}\left(1-\left(\frac{1-x+y}{1+x}\right)^{-\beta}\right)\frac{1}{|y|^{\alpha(x)+1}}dy\\
&=c(x)(1+\varepsilon)\int_{x}^{\infty}\left(1-\left(\frac{1-x+y}{1+x}\right)^{-\beta}\right)\frac{1}{|y|^{\alpha(x)+1}}dy\\&\
\ \ -2\varepsilon
c(x)\int_{2x}^{\infty}\left(1-\left(\frac{1-x+y}{1+x}\right)^{-\beta}\right)\frac{1}{|y|^{\alpha(x)+1}}dy.\end{align*}
Note that in the first inequality we make a change of variables
 $y\longmapsto-y$. Let us put
\begin{align*}U_5^{\varepsilon}(x):=&c(x)(1+\varepsilon)\int_{x}^{\infty}\left(1-\left(\frac{1-x+y}{1+x}\right)^{-\beta}\right)\frac{1}{|y|^{\alpha(x)+1}}dy\\&-2\varepsilon
c(x)\int_{2x}^{\infty}\left(1-\left(\frac{1-x+y}{1+x}\right)^{-\beta}\right)\frac{1}{|y|^{\alpha(x)+1}}dy.\end{align*}
We have
\be\label{eq:3.19}\displaystyle\int_{\{y+x<0\}}\left(1-\left(\frac{1-x-y}{1+x}\right)^{-\beta}\right)f_x(y)dy
>U_5^{\varepsilon}(x).\ee

\textbf{Step 4.} In the fourth step we prove
$$\liminf_{x\longrightarrow\infty}\frac{(1+x)^{\alpha(x)+\beta}}{c(x)}\left(\int_{\R}p(x,dy)V(y)-V(x)\right)>0.$$
 By combining
(\ref{eq:3.16}), (\ref{eq:3.17}), (\ref{eq:3.18}) and
(\ref{eq:3.19}) we have
\be\label{eq:3.20}(1+x)^{\beta}\left(\int_{\R}p(x,dy)V(y)-V(x)\right)>U^{\delta,\varepsilon}_1(x)+
U^{\delta,\varepsilon}_2(x)+U^{\delta}_3(x)+U^{\varepsilon}_4(x)+U_5^{\varepsilon}(x).\ee
In the rest of the fourth step we prove
\begin{eqnarray*}&&\liminf_{x\longrightarrow\infty}\frac{(1+x)^{\alpha(x)+\beta}}{c(x)}\left(\int_{\R}p(x,dy)V(y)-V(x)\right)\\ &&\begin{split}>&\liminf_{\delta\longrightarrow0}\liminf_{\varepsilon\longrightarrow0}\liminf_{x\longrightarrow\infty}\frac{(1+x)^{\alpha(x)}}{c(x)}U^{\delta,\varepsilon}_1(x)+\liminf_{\delta\longrightarrow0}\liminf_{\varepsilon\longrightarrow0}\liminf_{x\longrightarrow\infty}\frac{(1+x)^{\alpha(x)}}{c(x)}U^{\delta,\varepsilon}_2(x)\\&+\liminf_{\varepsilon\longrightarrow0}\liminf_{x\longrightarrow\infty}\frac{(1+x)^{\alpha(x)}}{c(x)}(U_4^{\varepsilon}(x)+U_5^{\varepsilon}(x))+T(\alpha,\beta)\geq0.\end{split}\end{eqnarray*}
 Recall
that $\limsup_{|x|\longrightarrow\infty}\alpha(x)=:\alpha<2$,
$$T(\alpha,\beta)=\sum_{n=1}^{\infty}{-\beta \choose
2n}\frac{2}{2n-\alpha}-\frac{2}{\alpha}+\frac{_2F_1\left(\beta,\alpha+\beta,1+\alpha+\beta;1\right)+\,
_2F_1\left(\beta,\alpha+\beta,1+\alpha+\beta;-1\right)}{\alpha+\beta}$$
and
$$\liminf_{\delta\longrightarrow0}\liminf_{x\longrightarrow\infty}\frac{(1+x)^{\alpha(x)}}{c(x)}U_3^{\delta}(x)>T(\alpha,\beta)$$
(assumption (\ref{eq:1.4})). Clearly,
\begin{align}\label{eq:3.21}&\liminf_{\delta\longrightarrow0}\liminf_{\varepsilon\longrightarrow0}\liminf_{x\longrightarrow\infty}\frac{(1+x)^{\alpha(x)}}{c(x)}U^{\delta,\varepsilon}_1(x)=0.\end{align}
  Further, by  the
dominated convergence theorem, we have
\begin{eqnarray*}&&\liminf_{\delta\longrightarrow0}\liminf_{\varepsilon\longrightarrow0}\liminf_{x\longrightarrow\infty}\frac{(1+x)^{\alpha(x)}}{c(x)}U^{\delta,\varepsilon}_2(x)\nonumber\\&&\begin{split}=\liminf_{\delta\longrightarrow0}\liminf_{\varepsilon\longrightarrow0}\liminf_{x\longrightarrow\infty}\Bigg[&-(1+\varepsilon)\displaystyle\sum_{i=2}^{\infty}{-\beta
\choose
i}\frac{(-1)^{i}}{i-\alpha(x)}\left(\left(\frac{x}{1+x}\right)^{i-\alpha(x)}-\delta^{i-\alpha(x)}\right)\\
&-\displaystyle\sum_{i=2}^{\infty}\left({-\beta \choose
i}\frac{1+(-1)^{i}\varepsilon}{i-\alpha(x)}\frac{\delta^{\alpha(x)}-\delta^{i}}{\delta^{\alpha(x)}}\right)\Bigg]\end{split}\nonumber\\
&&\begin{split}=\liminf_{\delta\longrightarrow0}\liminf_{\varepsilon\longrightarrow0}\liminf_{x\longrightarrow\infty}\Bigg[&-\displaystyle\sum_{i=2}^{\infty}{-\beta
\choose
i}\frac{(-1)^{i}\left(\frac{x}{1+x}\right)^{i-\alpha(x)}-(-1)^{i}\delta^{i-\alpha(x)}+1-\delta^{i-\alpha(x)}}{i-\alpha(x)}\\
&-\varepsilon\displaystyle\sum_{i=2}^{\infty}{-\beta \choose
i}\frac{(-1)^{i}\left(\frac{x}{1+x}\right)^{i-\alpha(x)}+(-1)^{i}-2(-1)^{i}\delta^{i-\alpha(x)}}{i-\alpha(x)}\Bigg]\end{split}\end{eqnarray*}

\begin{eqnarray}\label{eq:3.22}&&=\liminf_{\delta\longrightarrow0}\liminf_{x\longrightarrow\infty}-\displaystyle\sum_{i=2}^{\infty}{-\beta
\choose
i}\frac{(-1)^{i}\left(\frac{x}{1+x}\right)^{i-\alpha(x)}-(-1)^{i}\delta^{i-\alpha(x)}+1-\delta^{i-\alpha(x)}}{i-\alpha(x)}\nonumber\\
&&=\liminf_{x\longrightarrow\infty}-\displaystyle\sum_{i=2}^{\infty}{-\beta
\choose
i}\frac{(-1)^{i}\left(\frac{x}{1+x}\right)^{i-\alpha(x)}+1}{i-\alpha(x)}\nonumber\\
&&=-\displaystyle\sum_{i=1}^{\infty}{-\beta \choose
2i}\frac{2}{2i-\alpha}.\end{eqnarray} Now, let us compute
$$\liminf_{\varepsilon\longrightarrow0}\liminf_{x\longrightarrow\infty}\frac{(1+x)^{\alpha(x)}}{c(x)}(U^{\varepsilon}_4(x)+U^{\varepsilon}_5(x)).$$
By (\ref{eq:3.6}), we have
\begin{align}\label{eq:3.23}&\liminf_{\varepsilon\longrightarrow0}\liminf_{x\longrightarrow\infty}\frac{(1+x)^{\alpha(x)}}{c(x)}(U^{\varepsilon}_4(x)+U^{\varepsilon}_5(x))\nonumber\\
&=\liminf_{\varepsilon\longrightarrow0}\liminf_{x\longrightarrow\infty}\Bigg[(1-\varepsilon)\left(\frac{1}{\alpha(x)}-\frac{
_2F_1(\beta,\alpha(x)+\beta,1+\alpha(x)+\beta;-1)}{\alpha(x)+\beta}\right)\nonumber\\
&\ \ \
+(1+\varepsilon)\left(\frac{(x+1)^{\alpha(x)}}{\alpha(x)x^{\alpha(x)}}-\frac{(x+1)^{\alpha(x)+\beta}\
_2F_1\left(\beta,\alpha(x)+\beta,1+\alpha(x)+\beta;\frac{x-1}{x}\right)}{x^{\alpha(x)+\beta}(\alpha(x)+\beta)}\right)\nonumber\\
&\ \ \  -2\varepsilon
\left(\frac{(x+1)^{\alpha(x)}}{\alpha(x)(2x)^{\alpha(x)}}-\frac{(x+1)^{\alpha(x)+\beta}\
_2F_1\left(\beta,\alpha(x)+\beta,1+\alpha(x)+\beta;\frac{x-1}{2x}\right)}{x^{\alpha(x)+\beta}(\alpha(x)+\beta)2^{\alpha(x)+\beta}}\right)\Bigg]\nonumber\\
&=\liminf_{x\longrightarrow\infty}\left(\frac{2}{\alpha(x)}-\frac{_2F_1(\beta,\alpha(x)+\beta,1+\alpha(x)+\beta;-1)+\ _2F_1(\beta,\alpha(x)+\beta,1+\alpha(x)+\beta;1)}{\alpha(x)+\beta}\right)\nonumber\\
&=\frac{2}{\alpha}-\frac{_2F_1(\beta,\alpha+\beta,1+\alpha+\beta;-1)+\
_2F_1(\beta,\alpha+\beta,1+\alpha+\beta;1)}{\alpha+\beta},\end{align}
where in the second equality we use (\ref{eq:3.5}) and the fact that
all terms are bounded, and in the last equality we use the fact that
the function
$$x\longmapsto\frac{2}{x}-\frac{_2F_1(\beta,x+\beta,1+x+\beta;-1)+\
_2F_1(\beta,x+\beta,1+x+\beta;1)}{x+\beta}$$ is decreasing on
$(0,2)$. Now, by combining (\ref{eq:1.4}), (\ref{eq:3.20}),
(\ref{eq:3.21}), (\ref{eq:3.22}) and (\ref{eq:3.23}) we have
$$\liminf_{x\longrightarrow\infty}\frac{(1+x)^{\alpha(x)+\beta}}{c(x)}\left(\int_{\R}p(x,dy)V(y)-V(x)\right)>0.$$
The case when $x<0$ is treated in the same way.  Therefore,
 we have proved the desired result.
\end{proof}

\begin{proof}[Proof of Theorem \ref{tm1.3}]
In order to prove the theorem, according to \cite[Theorems 8.4.3 and
13.0.1]{meyn-tweedie-book}, Theorem \ref{tm1.1} and
\cite[Proposition 2.6]{sandric-rectrans}, it is enough to prove that
there exists a Borel measurable function
$V:\R\longrightarrow[0,\infty)$ such that corresponding level sets
$C_V(r):=\{y:V(y)\leq r\}$ are bounded  for every $r>0$ and there
exist $r_0>0$ and $d>0$, such that
$$\int_{\R}p(x,dy)V(y)-V(x)\leq-d$$ for all $x\in C^{c}_V(r_0)$ and
$$\sup\left\{\left|\int_{\R}p(x,dy)V(y)-V(x)\right|:x\in C_V(r_0)\right\}<\infty.$$

By assumption, $0<\inf\{\alpha(x):x\in\R\}$. Let
$\beta\in(0,1]\cap(0,\inf\{\alpha(x):x\in\R\})$ be arbitrary and for
the test function let us again take
$$V(x):=\log(1+|x|)\quad\textrm{and}\quad V(x):=|x|^{\beta}.$$
Since $\beta<\inf\{\alpha(x):x\in\R\}$, it easy to see that
$$\sup\left\{\left|\int_{\R}p(x,dy)V(y)-V(x)\right|:x\in C\right\}<\infty$$ for every
bounded set $C\subseteq\R$. Hence,  since $C_V(r)\uparrow\R$ for
$r\longrightarrow\infty$, it is enough to show that
$$\displaystyle\limsup_{|x|\longrightarrow\infty}\frac{|x|^{\alpha(x)
}}{c(x)}\left(\int_{\R}p(x,dy)V(y)-V(x)+d\right)<0$$ and
$$\displaystyle\limsup_{|x|\longrightarrow\infty}\frac{|x|^{\alpha(x)
-\beta}}{c(x)}\left(\int_{\R}p(x,dy)V(y)-V(x)+d\right)<0,$$ for some
$d>0$, respectively. Now, by performing completely the same
computations as in Theorem \ref{tm1.1} and using (\ref{eq:1.5}) and
(\ref{eq:1.6}), the desired result follows.
\end{proof}

\begin{proof}[Proof of Corollary \ref{c1.4}]
In the case when $\alpha\neq2$, the claim easily follows from
 Theorems \ref{tm1.1}
and \ref{tm1.2} and conditions (\ref{eq:1}) and (\ref{eq:2}), while,
in the case when $\alpha=2$,  the claim follows from
\cite[Proposition 8.5.4 ]{meyn-tweedie-book}.
\end{proof}

\begin{proof}[Proof of Theorem \ref{tm2.4}]
The claim of Theorem \ref{tm2.4} trivially follows from
\cite[Theorem 14.0.1]{meyn-tweedie-book}, (\ref{eq:2.4}),
(\ref{eq:2.5}) and Theorem \ref{tm1.3} by replacing the constant $1$
by an arbitrary Borel measurable  function $g(x)\geq1$ which
satisfies (\ref{eq:2.4}) or (\ref{eq:2.5}).
\end{proof}

\section{Some remarks  on the main results}

\quad \ \ We start this section with  the argumentation of various
versions of conditions (\ref{eq:1.2}), (\ref{eq:1.3}),
(\ref{eq:1.4}), (\ref{eq:1.5}) and (\ref{eq:1.6}) given in Theorems
\ref{tm1.1}, \ref{tm1.2} and \ref{tm1.3}.
 First, we show that under $$\limsup_{|x|\longrightarrow\infty}\alpha(x)<2\quad\textrm{and}\quad \lim_{|x|\longrightarrow\infty}c(x)|x|^{2-\alpha(x)}=\infty$$ conditions (\ref{eq:1.2}) and (\ref{eq:1.3})
are equivalent to (\ref{eq:1}). Indeed, from the elementary
inequalities
$$t-ct^{2}\leq\log(1+t)\leq t\quad \textrm{and}\quad \beta
t-dt^{2}\leq(1+t)^{\beta}-1\leq\beta t,$$ for $|t|$ small enough,
where $c>1/2$ and $d>\beta(1-\beta)/2$ are arbitrary, it follows
\begin{align*}&\mathrm {sgn}\it\,
(x)\frac{|\it{x}|^{\alpha(\it{x})}}{c(\it{x})(\rm{1}+|\it{x}|)}\int_{-\delta|\it{x}|}^{\delta|\it{x}|}yf_{\it{x}}(y)dy-c\frac{|\it{x}|^{\alpha(\it{x})}}{c(\it{x})(\rm{1}+|\it{x}|)^{\rm{2}}}\int_{-\delta|\it{x}|}^{\delta|\it{x}|}y^{\rm{2}}f_{\it{x}}(y)dy\nonumber\\
&\leq\frac{|\it{x}|^{\alpha(\it{x})}}{c(\it{x})}\int_{-\delta|\it{x}|}^{\delta|\it{x}|}\log\left(\rm{1}+\mathrm
{sgn}\it\,
(x)\frac{y}{\rm{1}+\it{|x|}}\right)f_{\it{x}}(y)dy\nonumber\\&\leq\mathrm
{sgn}\it\,
(x)\frac{|\it{x}|^{\alpha(\it{x})}}{c(\it{x})(\rm{1}+|\it{x}|)}\int_{-\delta|\it{x}|}^{\delta|\it{x}|}yf_{\it{x}}(y)dy\end{align*}
and \begin{align*}&\beta\,\mathrm {sgn}\it\,
(x)\frac{|\it{x}|^{\alpha(\it{x})-\rm{1}}}{c(\it{x})}\int_{-\delta|\it{x}|}^{\delta|\it{x}|}yf_{x}(y)dy-d\frac{|\it{x}|^{\alpha(\it{x})-\rm{2}}}{c(\it{x})}\int_{-\delta|\it{x}|}^{\delta|\it{x}|}y^{\rm{2}}f_{x}(y)dy\nonumber\\
&\leq\frac{|\it{x}|^{\alpha(\it{x})}}{c(\it{x})}\int_{-\delta|\it{x}|}^{\delta|\it{x}|}\left(\left(\rm{1}+\mathrm
{sgn}\it\,
(x)\frac{y}{\it{|x|}}\right)^{\beta}-\rm{1}\right)f_{\it{x}}(y)dy\nonumber\\
&\leq\beta\,\mathrm {sgn}\it\,
(x)\frac{|\it{x}|^{\alpha(\it{x})-\rm{1}}}{c(\it{x})}\int_{-\delta|\it{x}|}^{\delta|\it{x}|}yf_{\it{x}}(y)dy,\end{align*}
for $|x|$ large enough. Further, let $\varepsilon>0$ be arbitrary.
Then, by (C3), there exists $y_{\varepsilon}>0$ such that
\begin{align*}&\frac{|\it{x}|^{\alpha(\it{x})-\rm{2}}}{c(\it{x})}\int_{-\delta|\it{x}|}^{\delta|\it{x}|}y^{\rm{2}}f_{\it{x}}(y)dy\\&\leq
\frac{|\it{x}|^{\alpha(\it{x})-\rm{2}}}{c(\it{x})}\int_{-y_{\varepsilon}}^{y_{\varepsilon}}y^{\rm{2}}f_{\it{x}}(y)dy+2(1+\varepsilon)\left(\frac{\delta^{2-\alpha(x)}}{2-\alpha(x)}-\frac{|x|^{\alpha(x)-2}}{2-\alpha(x)}y_{\varepsilon}^{2-\alpha(x)}\right),\end{align*}
for $|x|$ large enough. Now, by taking
$\limsup_{\delta\longrightarrow0}\limsup_{\varepsilon\longrightarrow0}\limsup_{|x|\longrightarrow\infty}$,
it follows that
$$\limsup_{|x|\longrightarrow\infty}\frac{|\it{x}|^{\alpha(\it{x})-\rm{2}}}{c(\it{x})}\int_{-\delta|\it{x}|}^{\delta|\it{x}|}y^{\rm{2}}f_{\it{x}}(y)dy=0.$$
Hence, conditions (\ref{eq:1.2}) and (\ref{eq:1.3}) are equivalent
to
\begin{align}\label{eq:4.1}\limsup_{\delta\longrightarrow0}\limsup_{|x|\longrightarrow\infty}\mathrm
{sgn}(\it{x})\frac{|x|^{\alpha(x)-\rm{1}}}{c(x)}\int_{-\delta|x|}^{\delta|x|}yf_{x}(y)dy<R_{\rm{1}}(\alpha)
\end{align} and
\begin{align}\label{eq:4.2}\limsup_{\delta\longrightarrow0}\limsup_{|x|\longrightarrow\infty}\mathrm
{sgn}(\it{x})\frac{|x|^{\alpha(x)-\rm{1}}}{c(x)}\int_{-\delta|x|}^{\delta|x|}yf_{x}(y)dy<\frac{R_{\rm{2}}(\alpha,\beta)}{\beta},
\end{align} respectively.
Further, it can be proved that the function $\beta\longmapsto
R_2(\alpha,\beta)/\beta$ is strictly decreasing. Thus, in
(\ref{eq:4.2}), we choose $\beta$ close to zero. From
(\ref{eq:3.1}), (\ref{eq:3.2}) and (\ref{eq:3.4}), we have
$$\lim_{\beta\longrightarrow0}\frac{R_2(\alpha,\beta)}{\beta}=R_1(\alpha),$$
which proves the desired result. Next, the assumption
$\liminf_{|x|\longrightarrow\infty}\alpha(x)>1$ implies that the
transition densities $f_x(y)$ have finite first moment for all $|x|$
large enough. Therefore, in order to prove that (\ref{eq:1}) is
equivalent to the following drift condition
\be\limsup_{|x|\longrightarrow\infty}\mathrm{sgn}(\it{x})\frac{|\it{x}|^{\alpha(\it{x})-\rm{1}}}{c(\it{x})}\mathbb{E}[X_{\rm{1}}-X_{\rm{0}}|X_{\rm{0}}=\it{x}]<R_{\rm{1}}(\alpha)\ee
 it suffices to prove that
$$\limsup_{\delta\longrightarrow0}\limsup_{|x|\longrightarrow\infty}\mathrm {sgn}\it\,
(x)\frac{|\it{x}|^{\alpha(\it{x})-\rm{1}}}{c(\it{x})}\left(\int_{-\infty}^{-\delta|\it{x}|}yf_{\it{x}}(y)dy+\int_{\delta|\it{x}|}^{\infty}yf_{\it{x}}(y)dy\right)=\rm{0}.$$
But this fact can again be easily verified by using (C3).

Further, by completely the same arguments as above, it is easy to
check that, under
$$\limsup_{|x|\longrightarrow\infty}\alpha(x)<2\quad\textrm{and}\quad\lim_{|x|\longrightarrow\infty}c(x)|x|^{2-\alpha(x)}=\infty,$$
conditions (\ref{eq:1.5}) and (\ref{eq:1.6}) are equivalent to
(\ref{eq:3}) and (\ref{eq:4}), respectively.  Again, by having
finite first moments, that is, under
$\liminf_{|x|\longrightarrow\infty}\alpha(x)>1$, conditions
(\ref{eq:3}) and (\ref{eq:4}) are equivalent to the following drift
conditions
\be\limsup_{d\longrightarrow0}\limsup_{|x|\longrightarrow\infty}
\frac{|\it{x}|^{\alpha(\it{x})-\rm{1}}}{c(\it{x})}\left(\mathrm
{sgn}(\it{x})\mathbb{E}[X_{\rm{1}}-X_{\rm{0}}|X_{\rm{0}}=\it{x}]+d|\it{x}|\right)<R_{\rm{1}}(\alpha)\ee
and
\be\limsup_{\beta\longrightarrow0}\limsup_{d\longrightarrow0}\limsup_{|x|\longrightarrow\infty}\frac{|x|^{\alpha(x)-\rm{1}}}{c(x)}\left(\mathrm
{sgn}(\it{x})\mathbb{E}[X_{\rm{1}}-X_{\rm{0}}|X_{\rm{0}}=\it{x}]+\frac{d|x|^{-\beta+\rm{1}}}{\beta}\right)<R_{\rm{1}}(\alpha),\ee
respectively.

Finally, it remains to justify various versions of condition
(\ref{eq:1.4}). Under the assumption
$$\lim_{|x|\longrightarrow\infty}c(x)|x|^{2-\alpha(x)}=\infty$$
(recall that $\limsup_{|x|\longrightarrow\infty}\alpha(x)<2$ is
assumed in Theorem \ref{tm1.2}) and from the elementary inequality
$$\beta
t-ct^{2}\leq1-(1+t)^{-\beta}\leq\beta t,$$ for $|t|$ small enough,
where $c>\beta(\beta+1)/2$ is arbitrary, by completely the same
arguments as above, it follows that (\ref{eq:1.4}) is equivalent to
\be\liminf_{\delta\longrightarrow0}\liminf_{|x|\longrightarrow\infty}
\mathrm
{sgn}(\it{x})\frac{|\it{x}|^{\alpha(\it{x})-\rm{1}}}{c(\it{x})}\int_{-\delta|\it{x}|}^{\delta|\it{x}|}yf_{x}(y)dy>\frac{T(\alpha,\beta)}{\beta}.\ee
Further, it can be proved that the function $\beta\longmapsto
T(\alpha,\beta)/\beta$ is strictly increasing. According to this, we
choose $\beta$ close to zero. Again, from (\ref{eq:3.1}),
(\ref{eq:3.2}) and (\ref{eq:3.4}), we have
$$\lim_{\beta\longrightarrow0}\frac{T(\alpha,\beta)}{\beta}=R_1(\alpha),$$
hence (\ref{eq:1.4}) is equivalent to (\ref{eq:2}).  Again, by
assuming $1<\liminf_{|x|\longrightarrow\infty}\alpha(x)$,
(\ref{eq:1.4}) is equivalent to the following drift condition
\be\liminf_{|x|\longrightarrow\infty}\mathrm{sgn}(\it{x})\frac{|\it{x}|^{\alpha(\it{x})-\rm{1}}}{c(\it{x})}\mathbb{E}[X_{\rm{1}}-X_{\rm{0}}|X_{\rm{0}}=\it{x}]>R_{\rm{1}}(\alpha).\ee
 At the end, let us assume that
$\limsup_{|x|\longrightarrow\infty}\alpha(x)<1$. Then, it is not
hard to see that Theorem \ref{tm1.2} holds true under condition
\be\label{eq:4.8}\liminf_{\delta\longrightarrow0}\liminf_{|x|\longrightarrow\infty}\frac{\alpha(x)|x|^{\alpha(x)}}{c(x)}\int_{-\delta|x|}^{\delta|x|}\left(1-\left(1+\mathrm
{sgn}\it\,
(x)\frac{y}{\rm{1}+|\it{x}|}\right)^{-\beta}\right)f_{x}(y)dy>\alpha
T(\alpha,\beta).\ee Further, condition (\ref{eq:4.8}) is equivalent
to
\be\liminf_{a\longrightarrow\infty}\liminf_{|x|\longrightarrow\infty}\frac{\alpha(x)|x|^{\alpha(x)}}{c(x)}\int_{-a}^{a}\left(1-\left(1+\mathrm
{sgn}\it\,
(x)\frac{y}{\rm{1}+|\it{x}|}\right)^{-\beta}\right)f_{x}(y)dy>\alpha
T(\alpha,\beta).\ee Indeed, let $0<\varepsilon<1$ and $0<\delta<1$
be arbitrary. Then, by (C3) and (\ref{eq:3.6}), there exists
$y_{\varepsilon}>0$ such that
\begin{align*}&(\rm{1}-\mathrm {sgn}\it\,
(x)\varepsilon)|\it{x}|^{\alpha(\it{x})}\Bigg[\frac{\rm{1}}{y^{\alpha(x)}_{\varepsilon}}-\frac{\rm{1}}{\delta^{\alpha(x)}|\it{x}|^{\alpha(\it{x})}}+\frac{_{\rm{2}}F_{\rm{1}}\left(-\alpha(x),\beta,\rm{1}-\alpha(\it{x}),-\mathrm
{sgn}\it\,
(x)\frac{\delta|x|}{\rm{1}+|\it{x}|}\right)}{\delta^{\alpha(x)}|\it{x}|^{\alpha(\it{x})}}\\
&\ \ \
-\frac{_{\rm{2}}F_{\rm{1}}\left(-\alpha(\it{x}),\beta,\rm{1}-\alpha(\it{x}),-\mathrm
{sgn}\it\,
(x)\frac{y_{\varepsilon}}{\rm{1}+|\it{x}|}\right)}{y_{\varepsilon}^{\alpha(\it{x})}}\Bigg]+(\rm{1}+\mathrm
{sgn}\it\,
(x)\varepsilon)|\it{x}|^{\alpha(\it{x})}\Bigg[\frac{\rm{1}}{y^{\alpha(x)}_{\varepsilon}}
-\frac{\rm{1}}{\delta^{\alpha(x)}|\it{x}|^{\alpha(\it{x})}}\\
&\ \ \
+\frac{_{\rm{2}}F_{\rm{1}}\left(-\alpha(x),\beta,\rm{1}-\alpha(\it{x}),\mathrm
{sgn}\it\,
(x)\frac{\delta|x|}{\rm{1}+|\it{x}|}\right)}{\delta^{\alpha(x)}|\it{x}|^{\alpha(\it{x})}}-\frac{_{\rm{2}}F_{\rm{1}}\left(-\alpha(\it{x}),\beta,\rm{1}-\alpha(\it{x}),\mathrm
{sgn}\it\,
(x)\frac{y_{\varepsilon}}{\rm{1}+|\it{x}|}\right)}{y_{\varepsilon}^{\alpha(\it{x})}}\Bigg]\\
&\ \ \ +
\frac{|x|^{\alpha(x)}\alpha(x)}{c(x)}\int_{-y_{\varepsilon}}^{y_{\varepsilon}}\left(1-\left(1+\mathrm
{sgn}\it\,
(x)\frac{y}{\rm{1}+|\it{x}|}\right)^{-\beta}\right)f_{x}(y)dy\\
&\leq\frac{|x|^{\alpha(x)}\alpha(x)}{c(x)}\int_{-\delta|x|}^{\delta|x|}\left(1-\left(1+\mathrm
{sgn}\it\,
(x)\frac{y}{\rm{1}+|\it{x}|}\right)^{-\beta}\right)f_{x}(y)dy,\end{align*}
holds for all $|x|$ large enough. In the similar way we get a
similar upper bound for the left-hand side term in (\ref{eq:4.8}).
Now, since $\limsup_{|x|\longrightarrow\infty}\alpha(x)<1$, by
letting
$\liminf_{\delta\longrightarrow0}\liminf_{\varepsilon\longrightarrow0}\liminf_{y_{\varepsilon}\longrightarrow\infty}\liminf_{|x|\longrightarrow\infty}$
and applying (\ref{eq:3.5}), the desired result follows. Further,
from the concavity of the function $x\longmapsto x^{\beta}$ (recall
that  $\beta\in(0,1)$), we have
\begin{align*}&\frac{|x|^{\alpha(x)}\alpha(x)}{c(x)}\int_{-a}^{a}\left(1-\left(1+\mathrm
{sgn}\it \,(x)\frac{y}{\rm1+|\it x|}\right)^{-\beta}\right)f_x(y)dy\\
&\geq\frac{|x|^{\alpha(x)}\alpha(x)}{c(x)}\frac{(1+|x|-a)^{\beta}-(1+|x|)^{\beta}}{(1+|x|-a)^{\beta}}\int_{-a}^{a}f_x(y)dy\\
&\geq-\frac{a\beta\alpha(x)|x|^{\alpha(x)}}{c(x)(1+|x|-a)}.\end{align*}
Similarly,
$$\frac{|x|^{\alpha(x)}\alpha(x)}{c(x)}\int_{-a}^{a}\left(1-\left(1+\mathrm
{sgn}\it \,(x)\frac{y}{\rm1+|\it
x|}\right)^{-\beta}\right)f_x(y)dy\leq\frac{a\beta\alpha(x)|x|^{\alpha(x)}(1+|x|)^{\beta-1}}{c(x)(1+a+|x|)^{\beta}}.$$
Thus, by  taking
$0<\beta<1-\alpha=1-\limsup_{|x|\longrightarrow\infty}\alpha(x)$ and
letting
$\liminf_{a\longrightarrow\infty}\liminf_{|x|\longrightarrow\infty}$,
we get that condition  (\ref{eq:4.8}) follows from
\be\label{eq:4.10}\lim_{|x|\longrightarrow\infty}\frac{\alpha(x)|x|^{\alpha(x)-1}}{c(x)}=0\ee
(note that in this case $T(\alpha,\beta)<0$).  Essentially,
condition (\ref{eq:4.10}) says that the scaling function $c(x)$
cannot decrease too fast (recall that $\sup\{c(x):x\in\R\}<\infty$
(see \cite{sandric-rectrans})). Otherwise, we could enter in the
recurrence regime.

At the end, note that all conclusions, methods and proofs given in
this paper  can also be carried out in the discrete state space
$\ZZ$. Note that in this case conditions (C1)-(C5) are reduced just
to conditions  (C2) and (C3), since compact sets are replaced by
finite sets. Therefore, we deal with a Markov chain
$\{X^{d}_n\}_{n\geq0}$ on $\ZZ$ given by the transition kernel
$$p_{ij}:=f_i(j-i),$$ for $i,j\in\ZZ$, where $\{f_i:i\in\ZZ\}$ is a
family of probability functions which satisfies the following
conditions:
\begin{description}
\item [\textbf{(CD1)}] $f_i(j)\sim c(i)|j|^{-\alpha(i)-1},$
                          for
                          $|j|\longrightarrow\infty$, for every $i\in\ZZ$;
                          \item [\textbf{(CD2)}] there exists $k_0\in\N$ such that
\begin{align*}\lim_{|j|\longrightarrow\infty}\sup_{i\in\{-k_0,\ldots,k_0\}^{c}}\left|f_i(j)\frac{|j|^{\alpha(i)+1}}{c(i)}-1\right|=0.\end{align*}
\end{description}
Functions $\alpha:\ZZ\longrightarrow(0,2)$ and
$c:\ZZ\longrightarrow(0,\infty)$ are arbitrary given functions. The
proofs and assumptions of Theorems~\ref{tm1.1}, \ref{tm1.2},
\ref{tm1.3}, \ref{tm2.2}  and \ref{tm2.4} in the discrete case remain the same as
in the continuous case because we can switch from sums to integrals
due to the tail behavior of transition jumps.

 \section*{Acknowledgement} This work has been supported in part by Croatian Science Foundation under the project 3526. The author would like to thank    the anonymous reviewer for careful reading of the paper
and for helpful comments that led to improvement of the
presentation.

\bibliographystyle{alpha}
\bibliography{References}

\end{document}